\documentclass[preprint,10pt]{elsarticle}

\usepackage{amsmath}
\usepackage{amssymb}
\usepackage{amsthm}

\newtheorem{thm}{Theorem}[section]
\newtheorem{cor}[thm]{Corollary}
\newtheorem{lem}[thm]{Lemma}

\newtheorem{defin}[thm]{Definition}
\newtheorem{rem}[thm]{Remark}

\numberwithin{equation}{section}

\begin{document}

\begin{frontmatter}



\title{Weighted BMO-BLO estimates for Littlewood--Paley square operators}


\author{Runzhe Zhang and Hua Wang}
\address{School of Mathematics and System Science, Xinjiang University,\\
Urumqi 830046, P. R. China\\
\textbf{Dedicated to the memory of Li Xue}}
\ead{wanghua@pku.edu.cn}

\begin{abstract}
Let $T(f)$ denote the Littlewood--Paley square operators, including the Littlewood--Paley $\mathcal{G}$-function $\mathcal{G}(f)$, Lusin's area integral $\mathcal{S}(f)$ and Stein's function $\mathcal{G}^{\ast}_{\lambda}(f)$ with $\lambda>2$. We establish the boundedness of Littlewood--Paley square operators on the weighted spaces $\mathrm{BMO}(\omega)$ with $\omega\in A_1$. The weighted space $\mathrm{BLO}(\omega)$ (the space of functions with bounded lower oscillation) is introduced and studied in this paper. This new space is a proper subspace of $\mathrm{BMO}(\omega)$. It is proved that if $T(f)(x_0)$ is finite for a single point $x_0\in\mathbb R^n$, then $T(f)(x)$ is finite almost everywhere in $\mathbb R^n$. Moreover, it is shown that $T(f)$ is bounded from $\mathrm{BMO}(\omega)$ into $\mathrm{BLO}(\omega)$, provided that $\omega\in A_1$. The corresponding John--Nirenberg inequality suitable for the space $\mathrm{BLO}(\omega)$ with $\omega\in A_1$ is discussed. Based on this, the equivalent characterization of the space $\mathrm{BLO}(\omega)$ is also given.
\end{abstract}


\end{frontmatter}

\tableofcontents

\section{Introduction}
In the present paper, we study the existence and boundedness of Littlewood--Paley square operators, which are a generalization of the classical Littlewood--Paley operators defined by Poisson integrals.The weighted version of BLO spaces is introduced and studied, and weighted BMO-BLO estimates for Littlewood--Paley square operators are established. We believe that the results obtained in this paper are original in the weighted case.

The symbols $\mathbb R$ and $\mathbb N$ stand for the sets of all real numbers and natural numbers, respectively. Let $\mathbb R^n$ be the $n$-dimensional Euclidean space endowed with the Lebesgue measure $dx$. The Euclidean norm of $x=(x_1,x_2,\dots,x_n)\in \mathbb R^n$ is given by
\begin{equation*}
|x|:=\bigg(\sum_{i=1}^n x_i^2\bigg)^{1/2}.
\end{equation*}
It is well known that the Littlewood--Paley theory is a very important tool in harmonic analysis, complex analysis and PDEs. Littlewood--Paley theory can be viewed as a profound generalization of the Pythagorean theorem. It originated in the 1930's and developed in the 1960's. The Littlewood--Paley function in one dimension was first introduced by Littlewood and Paley in studying the dyadic decomposition of Fourier series (see \cite{littlewood,littlewood2,littlewood3}). The Littlewood--Paley function of higher dimensions was first defined and studied by Stein (see \cite{emstein1,emstein2,stein}). Let us now recall the classical Littlewood--Paley operators on $\mathbb R^n$, which include the $g$-function,Lusin's area integral and Stein's $g^{\ast}_{\lambda}$-function. Let $f$ be a locally integrable function in $\mathbb R^n$ and $u(x,t):=P_t*f(x)$ be the Poisson integral of $f$ on the upper half space ${\mathbb R}^{n+1}_+=\big\{(y,t)\in{\mathbb R}^{n+1}:y\in\mathbb R^n,t>0\big\}$, where
\begin{equation*}
P_t(x):=c_n\cdot\frac{t}{(t^2+|x|^2)^{{(n+1)}/2}}\quad \& \quad c_n=\frac{\Gamma((n+1)/2)}{\pi^{(n+1)/2}}
\end{equation*}
denotes the Poisson kernel in ${\mathbb R}^{n+1}_+$.Then the classical Littlewood--Paley $g$-function of $f$ is defined by
\begin{equation*}
g(f)(x):=\bigg(\int_0^\infty \big|\nabla u(x,t)\big|^2 t\,dt\bigg)^{1/2},
\end{equation*}
where
\begin{equation*}
\nabla=\Big(\frac{\partial}{\partial t},\frac{\partial}{\partial x_1},\dots,\frac{\partial}{\partial x_n}\Big)\quad \& \quad \big|\nabla u(x,t)\big|^2=\Big|\frac{\partial u}{\partial t}\Big|^2+\sum_{j=1}^n\Big|\frac{\partial u}{\partial x_j}\Big|^2.
\end{equation*}
The classical Lusin's area integral (also referred to as the square function) and Stein's $g^{\ast}_{\lambda}$-function are defined, respectively, by
\begin{equation*}
S(f)(x):=\bigg(\iint_{\Gamma(x)}\big|\nabla u(y,t)\big|^2 t^{1-n}\,dydt\bigg)^{1/2}
\end{equation*}
and
\begin{equation*}
g^*_\lambda(f)(x):=\bigg(\iint_{{\mathbb R}^{n+1}_+}\bigg(\frac t{t+|x-y|}\bigg)^{\lambda n}
\big|\nabla u(y,t)\big|^2 t^{1-n}\,dydt\bigg)^{1/2},
\end{equation*}
where $\lambda>2$, and $\Gamma(x)$ denotes the usual cone with the vertex $x$ in $\mathbb R^n$,
\begin{equation*}
\Gamma(x):=\Big\{(y,t)\in{\mathbb R}^{n+1}_+:|y-x|<t\Big\}~~\mbox{and}~~
{\mathbb R}^{n+1}_+=\mathbb R^n\times(0,+\infty).
\end{equation*}

\begin{itemize}
  \item By the Plancherel formula, we can easily see that the classical Littlewood--Paley operators are all bounded on $L^2(\mathbb R^n)$.
  \item Let $1<p<\infty$. Stein proved that the Littlewood--Paley $g$-function can characterize $L^p$ spaces. Moreover, there exist two positive constants $C_1$ and $C_2$, independent of $f$, such that
  \begin{equation}\label{11}
  C_1\|f\|_{L^p}\leq\|g(f)\|_{L^p}\leq C_2\|f\|_{L^p},
  \end{equation}
  for every $f\in L^p(\mathbb R^n)$. The above estimate also holds for Lusin's area integral $S(f)$ and Littlewood--Paley $g^{\ast}_{\lambda}$-function $g^*_\lambda(f)$ when $\lambda>2$. For the proofs of these results, see Stein \cite{emstein1,emstein2,stein} and Fefferman \cite{fe}.
\end{itemize}
It is well known that $u(x,t)=P_t*f(x)$ satisfies Laplace's equation $\Delta u=0$ in ${\mathbb R}^{n+1}_+$,
\begin{equation*}
\Delta u(x,t)=\Delta_x u(x,t)+\frac{\partial^2u(x,t)}{\partial t^2}=0,
\end{equation*}
and has boundary values equal to $f$, in the sense that
\begin{equation*}
\lim_{t\to 0}u(x,t)=\lim_{t\to 0}P_t*f(x)=f(x)
\end{equation*}
almost everywhere. Moreover, $u(\cdot,t)\rightarrow f(\cdot)$ in $L^p(\mathbb R^n)$ if $f\in L^p(\mathbb R^n)$ with $1\leq p<\infty$. The estimates for classical Littlewood--Paley operators rely heavily on some tricks from classical harmonic analysis and partial differential equations (Green's theorem, the mean value property of harmonic functions, etc).

We now consider the following more general Littlewood--Paley square operators on $\mathbb R^n$. Let $\psi$ be a real-valued function on $\mathbb R^n$ satisfying the following conditions.
\begin{itemize}
  \item[(\textbf{P1})] (\textbf{the vanishing condition}):
  \begin{equation}\label{psi1}
  \psi\in L^1(\mathbb R^n)\quad \mbox{and} \quad \int_{\mathbb R^n}\psi(x)\,dx=0.
  \end{equation}
  \item[(\textbf{P2})] (\textbf{the size condition}): there exist two positive constants $C_1$ and $\delta$ such that
  \begin{equation}\label{psi2}
  |\psi(x)|\leq C_1\cdot\frac{1}{(1+|x|)^{n+\delta}}.
  \end{equation}
  \item[(\textbf{P3})] (\textbf{the smoothness condition}): there exist two positive constants $C_2$ and $\gamma$ such that
  \begin{equation}\label{psi3}
  |\psi(x+h)-\psi(x)|\leq C_2\cdot\frac{|h|^{\gamma}}{(1+|x|)^{n+\delta+\gamma}},
  \end{equation}
  whenever $2|h|\leq |x|$.
\end{itemize}
For such a function $\psi$, the generalized Littlewood--Paley $g$-function, Lusin's area integral and Littlewood--Paley $g^{\ast}_{\lambda}$-function are defined as follows:
\begin{equation*}
\mathcal{G}(f)(x):=\bigg(\int_0^\infty \big|\psi_t*f(x)\big|^2\frac{dt}{t}\bigg)^{1/2},
\end{equation*}
\begin{equation*}
\mathcal{S}(f)(x):=\bigg(\iint_{\Gamma(x)}\big|\psi_t*f(y)\big|^2\frac{dydt}{t^{n+1}}\bigg)^{1/2},
\end{equation*}
and
\begin{equation*}
\mathcal{G}^{*}_{\lambda}(f)(x):=\bigg(\iint_{{\mathbb R}^{n+1}_+}\left(\frac t{t+|x-y|}\right)^{\lambda n}
\big|\psi_t*f(y)\big|^2\frac{dydt}{t^{n+1}}\bigg)^{1/2}, \quad \lambda>2,
\end{equation*}
where for any function $\psi$ and for any $t\in(0,+\infty)$, we denote
\begin{equation*}
\psi_t(x):=\frac{1}{t^n}\psi\Big(\frac{x}{\,t\,}\Big).
\end{equation*}
\begin{itemize}
  \item Denote by $LP$ the collection of all functions $\psi$ satisfying \eqref{psi1}, \eqref{psi2} and \eqref{psi3}.
  By the classical theory of vector-valued singular integral operators, we can also obtain the strong-type $(p,p)$ ($1<p<\infty$) and weak-type $(1,1)$ estimates for generalized (real-variable) Littlewood--Paley square operators, including the above generalized Littlewood--Paley $g$-function,Lusin's area integral and Littlewood--Paley $g^{\ast}_{\lambda}$-function (these operators are sublinear and non-negative).
  \item As in \eqref{11}, it was shown that the generalized Littlewood--Paley $g$-function $\mathcal{G}$ can also characterize $L^p$ spaces. For any $1<p<\infty$ and $\psi\in LP$, then there exist two positive constants $C_1$ and $C_2$, independent of $f$, such that
  \begin{equation}\label{12}
  C_1\|f\|_{L^p}\leq\|\mathcal{G}(f)\|_{L^p}\leq C_2\|f\|_{L^p},
  \end{equation}
  for all $f\in L^p(\mathbb R^n)$. Moreover, the above estimate also holds for generalized Lusin's area integral $\mathcal{S}(f)$ and Stein's function $\mathcal{G}^*_{\lambda}(f)$ when $\lambda>2$. More details can be found in
  \cite[Chapter 5]{lu}, \cite[Chapter XII]{tor}, \cite[Chapter 6]{wilson2} and \cite[Theorem 1.1]{xueding}.
\end{itemize}

\begin{rem}\label{rem11}
If we take
\begin{equation*}
\psi(x)=\frac{\partial}{\partial t}P_t(x)\Big|_{t=1},
\end{equation*}
then by the well-known properties of the Poisson kernel $P_t(x)$ and the mean value theorem, we can easily check that the above three conditions are satisfied.
\end{rem}

Recall that, for any given $p\in[1,+\infty)$, the Lebesgue space $L^p(\mathbb R^n)$ is defined as the set of all integrable functions $f$ on $\mathbb R^n$ such that
\begin{equation*}
\|f\|_{L^p}:=\bigg(\int_{\mathbb R^n}|f(x)|^p\,dx\bigg)^{1/p}<+\infty,
\end{equation*}
Let $L^{\infty}(\mathbb R^n)$ denote the Banach space of all essentially bounded measurable functions $f$ on $\mathbb R^n$.
The norm of $f\in L^{\infty}(\mathbb R^n)$ is given by
\begin{equation*}
\|f\|_{L^\infty}:=\underset{x\in\mathbb R^n}{\mbox{ess\,sup}}\,|f(x)|<+\infty.
\end{equation*}
For any $x_0\in\mathbb R^n$ and $r>0$, let $B(x_0,r):=\{x\in\mathbb R^n:|x-x_0|<r\}$ denote the open ball centered at $x_0$ with the radius $r$ and $B(x_0,r)^{\complement}$ denote its complement. Given $B=B(x_0,r)$ and $t>0$, we will write $tB$ for the $t$-dilate ball, which is the ball with the same center $x_0$ and with radius $tr$.Similarly, $Q(x_0,r)$ denotes the cube centered at $x_0$ and with the sidelength $r$. Here and in what follows, only cubes with sides parallel to the coordinate axes are considered, and $t Q=Q(x_0,t r)$. For a measurable set $E\subset\mathbb R^n$, we use the notation $m(E)$ for the $n$-dimensional Lebesgue measure of the set $E$, and we use the notation $\chi_{E}$ to denote the characteristic function of the set $E$: $\chi_E(x)=1$ if $x\in E$ and $0$ if $x\notin E$.

Let us now recall the definition of the space $\mathrm{BMO}(\mathbb R^n)$. A locally integrable function $f$ on $\mathbb R^n$ is said to be in $\mathrm{BMO}(\mathbb R^n)$, the space of bounded mean oscillation (see \cite{john}), if
\begin{equation*}
\|f\|_{\mathrm{BMO}}:=\sup_{\mathcal{B}\subset\mathbb R^n}\frac{1}{m(\mathcal{B})}
\int_{\mathcal{B}}|f(x)-f_{\mathcal{B}}|\,dx<+\infty,
\end{equation*}
where $f_{\mathcal{B}}$ denotes the mean value of $f$ over $\mathcal{B}$, i.e.,
\begin{equation*}
f_{\mathcal{B}}:=\frac{1}{m(\mathcal{B})}\int_{\mathcal{B}} f(y)\,dy
\end{equation*}
and the supremum is taken over all balls $\mathcal{B}$ in $\mathbb R^n$. Modulo constants, the space $\mathrm{BMO}(\mathbb R^n)$ is a Banach space with respect to the norm $\|\cdot\|_{\mathrm{BMO}}$. The space of BMO functions was first introduced by John and Nirenberg in \cite{john}.

A locally integrable function $f$ on $\mathbb R^n$ is said to be in $\mathrm{BLO}(\mathbb R^n)$, the space of bounded lower oscillation (see \cite{coifman}), if there exists a constant $C>0$ such that for any ball $\mathcal{B}\subset\mathbb R^n$,
\begin{equation*}
\frac{1}{m(\mathcal{B})}\int_{\mathcal{B}}\Big[f(x)-\underset{y\in\mathcal{B}}{\mathrm{ess\,inf}}\,f(y)\Big]\,dx\leq C.
\end{equation*}
The minimal constant $C$ as above is defined to be the BLO-constant of $f$ and denoted by $\|f\|_{\mathrm{BLO}}$. The space of BLO functions was first introduced by Coifman and Rochberg in \cite{coifman}. It is easy to see that
\begin{equation*}
L^{\infty}(\mathbb R^n)\subset\mathrm{BLO}(\mathbb R^n)\subset\mathrm{BMO}(\mathbb R^n).
\end{equation*}
Moreover, the above inclusion relations are both strict, see \cite{hu,meng,ou} for some examples.

\begin{rem}
Equivalently, we could define the above notions with cubes instead of balls. Hence we shall use these two different definitions appropriate to calculations.
\end{rem}

A few historical remarks are given as follows:
\begin{enumerate}
  \item The Littlewood--Paley $g$-function of $f$ is well defined under some assumptions on $f$, although it may be infinite on a set of positive measure. In 1985, Wang \cite{wang} first studied the behavior of classical Littlewood--Paley $g$-function acting on $L^{\infty}(\mathbb R^n)$ and $\mathrm{BMO}(\mathbb R^n)$, and proved the following result. If $f\in \mathrm{BMO}(\mathbb R^n)$, then $g(f)$ is either
infinite everywhere or finite almost everywhere, and in the latter case, there is a positive constant $C$ depending only on the dimension $n$ such that
\begin{equation*}
\big\|g(f)\big\|_{\mathrm{BMO}}\leq C\big\|f\big\|_{\mathrm{BMO}}.
\end{equation*}
The above interesting result also holds for the classical Lusin's area integral $S(f)$ and Stein's $g^{\ast}_{\lambda}$-function $g^{\ast}_{\lambda}(f)$, which was established by Kurtz \cite{k} in 1987.
  \item Subsequently, in 2004, Sun \cite{sun} and Yabuta \cite{ya} studied the existence and boundedness properties of generalized Littlewood--Paley operators on BMO spaces (and Campanato spaces), and proved the following result. Suppose that $\psi\in L^1(\mathbb R^n)$ satisfies \eqref{psi1}, \eqref{psi2} with $\delta=1$, and the condition
\begin{equation}\label{psi4}
\big|\nabla\psi(x)\big|\leq C_2\cdot\frac{1}{(1+|x|)^{n+2}},
\end{equation}
where $\nabla:=(\partial/{\partial x_1},\dots,\partial/{\partial x_n})$ and $C_2$ is a positive constant independent of $x=(x_1,\dots,x_n)\in\mathbb R^n$. Then the generalized Littlewood--Paley $g$-function $\mathcal{G}$ is bounded on $\mathrm{BMO}(\mathbb R^n)$. More precisely, if $f\in \mathrm{BMO}(\mathbb R^n)$ and $\mathcal{G}(f)(x_0)<+\infty$ for a single point $x_0\in \mathbb R^n$, then $\mathcal{G}(f)$ is finite almost everywhere, and there exists a positive constant $C>0$, independent of $f$, such that
\begin{equation*}
\big\|\mathcal{G}(f)\big\|_{\mathrm{BMO}}\leq C\big\|f\big\|_{\mathrm{BMO}}.
\end{equation*}
Similar results for generalized Lusin's area integral and Littlewood--Paley $g^{\ast}_{\lambda}$-function were also obtained in \cite{sun,ya}.
\end{enumerate}

In 1990, Leckband \cite{leck} established the boundedness of the square of the Littlewood--Paley $g$-function, Lusin's area integral and Littlewood--Paley $g^{\ast}_{\lambda}$-function from $L^{\infty}(\mathbb R^n)$ into $\mathrm{BLO}(\mathbb R^n)$, which is a proper subspace of $\mathrm{BMO}(\mathbb R^n)$. More precisely, Leckband proved that (see \cite[Theorem 1]{leck})

\begin{thm}
If $f\in L^{\infty}(\mathbb R^n)$, then there exists a positive constant $C>0$, independent of $f$, such that
\begin{equation*}
\big\|[T_{g}(f)]^2\big\|_{\mathrm{BLO}}\leq C\big\|f\big\|_{L^{\infty}}^2,
\end{equation*}
where $T_{g}(f)$ denotes any one of the usual classical or generalized Littlewood--Paley functions.
\end{thm}

In 2008, Meng and Yang \cite{meng} further discussed the behavior of generalized Littlewood--Paley operators on BMO spaces. Let $\mathcal{G}(f)$ be the Littlewood--Paley $g$-function of $f$ on $\mathbb R^n$. Meng and Yang proved that if $f\in \mathrm{BMO}(\mathbb R^n)$, then $\mathcal{G}(f)$ is either infinite everywhere or finite almost everywhere, and in the latter case, $[\mathcal{G}(f)]^2$ is bounded from $\mathrm{BMO}(\mathbb R^n)$ into $\mathrm{BLO}(\mathbb R^n)$(see \cite[Theorem 1.1]{meng}), which is an improvement of the result of Leckband.

\begin{thm}
Suppose that $\psi\in L^1(\mathbb R^n)$ satisfies \eqref{psi1}, \eqref{psi2} with $\delta=1$ and \eqref{psi4}. If $f\in \mathrm{BMO}(\mathbb R^n)$, then $\mathcal{G}(f)$ is either infinite everywhere or finite almost everywhere, and in the latter case, there exists a positive constant $C>0$, independent of $f$, such that
\begin{equation*}
\big\|[\mathcal{G}(f)]^2\big\|_{\mathrm{BLO}}\leq C\big\|f\big\|_{\mathrm{BMO}}^2.
\end{equation*}
\end{thm}
Similar results for generalized Lusin's area integral and Littlewood--Paley $g^{\ast}_{\lambda}$-function were also obtained by Meng and Yang (see \cite[Theorems 1.2 and 1.3]{meng}). The corresponding estimates for Marcinkiewicz integrals can be found in \cite{hu}.

We remark that the condition \eqref{psi4} implies \eqref{psi3} with $\gamma=1$, by applying the mean value theorem. Arguing as in the proof of Theorem 1.1 in \cite{meng}, we can also show that the conclusions of the above theorem still hold for the generalized Littlewood--Paley operators $\mathcal{G}(f)$,$\mathcal{S}(f)$ and $\mathcal{G}^{\ast}_{\lambda}(f)$, under the conditions \eqref{psi1},\eqref{psi2} and \eqref{psi3} on $\psi$.

Note that for any given ball $\mathcal{B}$ in $\mathbb R^n$ and for any $x\in \mathcal{B}$, if
\begin{equation*}
\underset{y\in\mathcal{B}}{\mathrm{ess\,inf}}\,\big[\mathcal{F}(y)\big]<+\infty,
\end{equation*}
then
\begin{equation*}
\big[\mathcal{F}(x)\big]-\underset{y\in\mathcal{B}}{\mathrm{ess\,inf}}\,\big[\mathcal{F}(y)\big]\leq
\Big(\big[\mathcal{F}(x)\big]^2-\underset{y\in\mathcal{B}}{\mathrm{ess\,inf}}\,\big[\mathcal{F}(y)\big]^2\Big)^{1/2},
\end{equation*}
which in turn implies that
\begin{equation}\label{BLOsquare}
\big\|\mathcal{F}\big\|^2_{\mathrm{BLO}}\leq\big\|\mathcal{F}^2\big\|_{\mathrm{BLO}}.
\end{equation}

As a direct consequence of \eqref{BLOsquare}, we obtain the following results.
\begin{thm}\label{thm11}
If $f\in L^{\infty}(\mathbb R^n)$, then there exists a positive constant $C>0$, independent of $f$, such that
\begin{equation*}
\big\|T_{g}(f)\big\|_{\mathrm{BLO}}\leq C\big\|f\big\|_{L^{\infty}},
\end{equation*}
where $T_{g}(f)$ denotes any one of the usual classical or generalized Littlewood--Paley functions.
\end{thm}

\begin{thm}\label{thm12}
Suppose that $\psi\in L^1(\mathbb R^n)$ satisfies \eqref{psi1}, \eqref{psi2} with $\delta=1$ and \eqref{psi4}. If $f\in \mathrm{BMO}(\mathbb R^n)$, then $\mathcal{G}(f)$ is either infinite everywhere or finite almost everywhere, and in the latter case, there exists a positive constant $C>0$, independent of $f$, such that
\begin{equation*}
\big\|\mathcal{G}(f)\big\|_{\mathrm{BLO}}\leq C\big\|f\big\|_{\mathrm{BMO}}.
\end{equation*}
\end{thm}

\section{Definitions of weighted BMO and BLO spaces}
A weight $\omega$ is a nonnegative locally integrable function on $\mathbb R^n$ that takes values in $(0,+\infty)$ almost everywhere.
For $1<p<\infty$, a weight $\omega$ is said to belong to the Muckenhoupt class $A_p$, if there exists a positive constant $C$ such that
\begin{equation*}
\bigg(\frac{1}{m(\mathcal{B})}\int_{\mathcal{B}}\omega(x)\,dx\bigg)
\bigg(\frac{1}{m(\mathcal{B})}\int_{\mathcal{B}}\omega(x)^{1-p'}dx\bigg)^{p-1}\leq C<+\infty
\end{equation*}
holds for every ball $\mathcal{B}$ in $\mathbb R^n$, where $p'=p/{(p-1)}$ denotes the H\"{o}lder conjugate exponent of $p$. For $p=1$, a weight $\omega$ is said to belong to the Muckenhoupt class $A_1$, if there is a positive constant $C$ such that
\begin{equation*}
\frac{1}{m(\mathcal{B})}\int_{\mathcal{B}}\omega(x)\,dx\leq C\cdot\underset{x\in \mathcal{B}}{\mbox{ess\,inf}}\,\omega(x)
\end{equation*}
holds for every ball $B$ in $\mathbb R^n$. The smallest constant $C>0$ such that the above inequality holds is called the $A_1$-constant of $\omega$, and is denoted by $[\omega]_{A_1}$. It is well known that $A_1\subset A_p$, for any $1<p<\infty$. Given a Lebesgue measurable set $E\subseteq\mathbb R^n$, we use the notation
\begin{equation*}
\omega(E):=\int_{E}\omega(x)\,dx
\end{equation*}
to denote the $\omega$-measure of the set $E$. Then we have
\begin{equation}\label{A1property}
\frac{1}{m(\mathcal{B})}\int_{\mathcal{B}}\omega(x)\,dx=\frac{\omega(\mathcal{B})}{m(\mathcal{B})}
\leq [\omega]_{A_1}\cdot\underset{x\in \mathcal{B}}{\mbox{ess\,inf}}\,\omega(x).
\end{equation}
Given a weight $\omega$ defined on $\mathbb R^n$, as usual, the weighted Lebesgue space $L^p(\omega)$ for $1\leq p<\infty$ is defined as the set of all integrable functions $f$ on $\mathbb R^n$ such that
\begin{equation*}
\big\|f\big\|_{L^p(\omega)}:=\bigg(\int_{\mathbb R^n}|f(x)|^p\omega(x)\,dx\bigg)^{1/p}<+\infty.
\end{equation*}

\begin{defin}
Let $\omega$ be a weight function on $\mathbb R^n$.We define the weighted $\mathrm{BMO}$ space as follows:
\begin{equation*}
\mathrm{BMO}(\omega):=\Big\{f\in L^1_{\mathrm{loc}}(\mathbb R^n):\|f\|_{\mathrm{BMO}(\omega)}<+\infty\Big\},
\end{equation*}
where
\begin{equation*}
\|f\|_{\mathrm{BMO}}:=\sup_{\mathcal{B}\subset\mathbb R^n}\frac{1}{\omega(\mathcal{B})}
\int_{\mathcal{B}}|f(x)-f_{\mathcal{B}}|\,dx,
\end{equation*}
and $f_{\mathcal{B}}$ denotes the mean value of the function $f$ over $\mathcal{B}$.
\end{defin}
The weighted space $\mathrm{BMO}(\omega)$ was introduced and studied by Garcia-Cuerva in \cite{garcia1}. In general, for $1\leq p<\infty$, a locally integrable function $f$ is said to be in $\mathrm{BMO}^p(\omega)$, if
\begin{equation*}
\|f\|_{\mathrm{BMO}^{p}(\omega)}:=\sup_{\mathcal{B}}\left(\frac{1}{\omega(\mathcal{B})}\int_{\mathcal{B}}
\big|f(x)-f_{\mathcal{B}}\big|^{p}\omega(x)^{1-p}dx\right)^{1/p}<+\infty.
\end{equation*}
Let $\omega\in A_1$. Garcia-Cuerva in \cite{garcia1} proved that for any $1\leq p<\infty$, there exists an absolute constant $C>0$ such that
\begin{equation*}
\|f\|_{\mathrm{BMO}^{p}(\omega)}\leq C\|f\|_{\mathrm{BMO}(\omega)}.
\end{equation*}
This tells us that the spaces $\mathrm{BMO}^p(\omega)$ coincide, and the norms of $\|\cdot\|_{\mathrm{BMO}^{p}(\omega)}$ are equivalent with respect to different values of $p$, provided that $\omega\in A_1$. See Theorem \ref{lastthm2} in the last section. For more results about weighted BMO spaces, we refer the readers to \cite{bloom,hubei,wang1,wang2} and the references therein.

Motivated by the definition of $\mathrm{BMO}(\omega)$, we now introduce the following space $\mathrm{BLO}(\omega)$, which is a subspace of $\mathrm{BMO}(\omega)$.

\begin{defin}
Let $\omega$ be a weight function on $\mathbb R^n$. We say that a locally integrable function $f$ on $\mathbb R^n$ is in the space $\mathrm{BLO}(\omega)$, if there exists a constant $C>0$ such that for any ball $\mathcal{B}\subset\mathbb R^n$,
\begin{equation*}
\frac{1}{\omega(\mathcal{B})}\int_{\mathcal{B}}\Big[f(x)-\underset{y\in\mathcal{B}}{\mathrm{ess\,inf}}\,f(y)\Big]\,dx\leq C.
\end{equation*}
The minimal constant $C$ as above is defined to be the $\mathrm{BLO}$-constant of $f$ with respect to $\omega$, and is denoted by $\|f\|_{\mathrm{BLO}(\omega)}$.
\end{defin}
It is easy to verify that
\begin{equation}\label{bmoblo}
\|f\|_{\mathrm{BMO}(\omega)}\leq 2\|f\|_{\mathrm{BLO}(\omega)}.
\end{equation}
In fact, for any ball $\mathcal{B}\subset\mathbb R^n$,
\begin{equation*}
\begin{split}
&\frac{1}{\omega(\mathcal{B})}\int_{\mathcal{B}}\big|f(x)-f_{\mathcal{B}}\big|\,dx\\
&=\frac{1}{\omega(\mathcal{B})}\int_{\mathcal{B}}\Big|f(x)-\underset{y\in\mathcal{B}}{\mathrm{ess\,inf}}\,f(y)
+\underset{y\in\mathcal{B}}{\mathrm{ess\,inf}}\,f(y)-f_{\mathcal{B}}\Big|\,dx\\
&\leq\frac{1}{\omega(\mathcal{B})}\int_{\mathcal{B}}\Big[f(x)-\underset{y\in\mathcal{B}}{\mathrm{ess\,inf}}\,f(y)\Big]\,dx
+\frac{m(\mathcal{B})}{\omega(\mathcal{B})}\Big|\underset{y\in\mathcal{B}}{\mathrm{ess\,inf}}\,f(y)-f_{\mathcal{B}}\Big|\\
&\leq\frac{2}{\omega(\mathcal{B})}\int_{\mathcal{B}}\Big[f(x)-\underset{y\in\mathcal{B}}{\mathrm{ess\,inf}}\,f(y)\Big]\,dx
\leq 2\|f\|_{\mathrm{BLO}(\omega)},
\end{split}
\end{equation*}
as desired.
\begin{rem}
\begin{itemize}
  \item It should be pointed out that $\|\cdot\|_{\mathrm{BLO}(\omega)}$ is not a norm and $\mathrm{BLO}(\omega)$ is not a linear space (it is a proper subspace of $\mathrm{BMO}(\omega)$).
  \item For the weighted case, we can also define the above notions with cubes in place of balls, and these two different definitions are actually equivalent.
\end{itemize}
\end{rem}
Similarly, we can also define the general version of weighted BLO spaces. For $1\leq p<\infty$, we say that a locally integrable function $f$ is in $\mathrm{BLO}^p(\omega)$, if
\begin{equation*}
\|f\|_{\mathrm{BLO}^{p}(\omega)}:=\sup_{\mathcal{B}}\left(\frac{1}{\omega(\mathcal{B})}\int_{\mathcal{B}}
\Big[f(x)-\underset{y\in\mathcal{B}}{\mathrm{ess\,inf}}\,f(y)\Big]^{p}\omega(x)^{1-p}dx\right)^{1/p}<+\infty.
\end{equation*}
In this paper, we will establish a weighted version of John--Nirenberg's inequality suitable for the space $\mathrm{BLO}(\omega)$ with $\omega\in A_1$. There exist two positive constants $C_1$ and $C_2$, depending only on the dimension $n$, such that for every cube $\mathcal{Q}$ in $\mathbb R^n$ and every $\lambda>0$,
\begin{equation*}
\begin{split}
&\nu\Big(\Big\{x\in \mathcal{Q}:\Big[f(x)-\underset{y\in\mathcal{Q}}{\mathrm{ess\,inf}}\,f(y)\Big]>\lambda\Big\}\Big)\\
&\leq {C}_1\cdot \nu(\mathcal{Q})\exp\bigg\{-\Big[[\omega]_{A_1}\underset{x\in\mathcal{Q}}{\mathrm{ess\,inf}}\,\omega(x)\Big]^{-1}
\frac{{C}_2\lambda}{\|f\|_{\mathrm{BLO}(\omega)}}\bigg\},
\end{split}
\end{equation*}
where $\nu(x):=\omega(x)^{1-p}$ for $1<p<\infty$. The proof is based on the Calder\'{o}n--Zygmund decomposition and stopping time arguments. We shall use this inequality to prove that the spaces $\mathrm{BLO}^p(\omega)$ coincide, and the $\mathrm{BLO}$-constants of $\|\cdot\|_{\mathrm{BLO}^{p}(\omega)}$ are equivalent with respect to different values of $p$, provided that $\omega\in A_1$. For further details, see Lemma \ref{expblo} and Theorem \ref{weightedblomain} in the last section.

The main purpose of this paper is to establish weighted BMO-BLO estimates for classical and generalized Littlewood--Paley square operators. We also give the corresponding $L^{\infty}$-BLO results in the weighted case.

\section{Main results}
By the weighted theory of vector-valued singular integral operators, we can deduce the following weighted $L^p$-boundedness of Littlewood--Paley square operators. The proof of such kind of results can be found in \cite[Chapter 5]{lu} and \cite[Chapter 7]{wilson2} (see also \cite{xueding} and \cite{kurtz}).

\begin{thm}\label{thmgfunction}
For all $1<p<\infty$ and $\omega\in A_p$, then there exists a positive constant $C>0$, independent of $f$, such that
\begin{equation*}
\big\|\mathcal{G}({f})\big\|_{L^p(\omega)}\leq C\big\|f\big\|_{L^p(\omega)}
\end{equation*}
for every $f\in L^p(\omega)$.
\end{thm}

\begin{thm}\label{thmsfunction}
For all $1<p<\infty$ and $\omega\in A_p$, then there exists a positive constant $C>0$, independent of $f$, such that
\begin{equation*}
\big\|\mathcal{S}({f})\big\|_{L^p(\omega)}\leq C\big\|f\big\|_{L^p(\omega)}
\end{equation*}
for every $f\in L^p(\omega)$.
\end{thm}
These weighted $L^p$ estimates were extended by many authors to the multilinear case, one can see \cite{shi,xueqing,xuepeng} for more details.

Inspired by Theorem \ref{thm11} and Theorem \ref{thm12}, we consider the behavior on weighted $L^{\infty}$ and weighted BMO spaces for Littlewood--Paley square operators. Based on Theorem \ref{thmgfunction} and Theorem \ref{thmsfunction}, as well as some properties of $A_p$ weights and weighted BMO spaces, we establish the existence and boundedness of generalized Littlewood--Paley square operators on weighted BMO spaces, including the $g$-function, Lusin's area integral and Stein's $g^{\ast}_{\lambda}$-function for $\lambda>3+{(2\delta+2\gamma)}/n$. It is proved that if the above operators are finite for one point, then they are finite almost everywhere in $\mathbb R^n$. Moreover, these operators are bounded from $\mathrm{BMO}(\omega)$ into $\mathrm{BLO}(\omega)$ when $\omega\in A_1$. We will also give the $L^{\infty}$-BLO results for Littlewood--Paley square operators, such kind of results can be viewed as a generalization of Theorem \ref{thm11} in the weighted case.

The main results of this paper are stated as follows.

\begin{thm}\label{mainthm1}
For any ${f}\in \mathrm{BMO}(\omega)$ and $\omega\in A_1$, then $\mathcal{G}({f})$ is either infinite everywhere or finite almost everywhere, and in the latter case, there exists a positive constant $C$, independent of ${f}$, such that
\begin{equation*}
\big\|\mathcal{G}({f})\big\|_{\mathrm{BLO}(\omega)}\leq C\big\|f\big\|_{\mathrm{BMO}(\omega)}.
\end{equation*}
\end{thm}

\begin{thm}\label{mainthm2}
For any ${f}\in \mathrm{BMO}(\omega)$ and $\omega\in A_1$, then $\mathcal{S}({f})$ is either infinite everywhere or finite almost everywhere, and in the latter case, there exists a positive constant $C$, independent of ${f}$, such that
\begin{equation*}
\big\|\mathcal{S}({f})\big\|_{\mathrm{BLO}(\omega)}\leq C\big\|f\big\|_{\mathrm{BMO}(\omega)}.
\end{equation*}
\end{thm}

\begin{thm}\label{mainthm3}
Suppose that $(\lambda-3)n>2\delta+2\gamma$. For any ${f}\in \mathrm{BMO}(\omega)$ and $\omega\in A_1$, then $\mathcal{G}^{\ast}_{\lambda}({f})$ is either infinite everywhere or finite almost everywhere, and in the latter case, there exists a positive constant $C$, independent of ${f}$, such that
\begin{equation*}
\big\|\mathcal{G}^{\ast}_{\lambda}({f})\big\|_{\mathrm{BLO}(\omega)}\leq C\big\|f\big\|_{\mathrm{BMO}(\omega)}.
\end{equation*}
\end{thm}

By remark \ref{rem11}, we may take $\psi(x)$ to be
\begin{equation*}
\psi(x)=\frac{\partial}{\partial t}P_t(x)\Big|_{t=1},
\end{equation*}
and then the functions $\mathcal{G}({f})(x)$, $\mathcal{S}({f})(x)$ and $\mathcal{G}^{\ast}_{\lambda}({f})(x)$ defined above become the classical Littlewood--Paley functions. Hence, the conclusions of the above theorems also hold for the classical Littlewood--Paley operators.

Throughout this paper, the letter $C$ will stand for a positive constant, not necessarily the same at each occurrence but independent of the essential variables. By $\mathbf{X}\lesssim\mathbf{Y}$, we mean that there exists a positive constant $C>0$ such that $\mathbf{X}\leq C\mathbf{Y}$. If $\mathbf{X}\lesssim\mathbf{Y}$ and $\mathbf{Y}\lesssim\mathbf{X}$, then we write $\mathbf{X}\approx\mathbf{Y}$ and say that $\mathbf{X}$ and $\mathbf{Y}$ are equivalent.

\section{Some key lemmas}
Let us begin by recalling some well-known results about $A_1$ weights (see, for example, \cite{duoand} and \cite{grafakos2}).
\begin{lem}\label{lemmaA1}
Let $n\in\mathbb N$. Then we have:
\begin{enumerate}
  \item[$(1)$] for $\omega\in A_1$ and for any ball $\mathcal{B}$ in $\mathbb R^n$,
\begin{equation*}
\omega(2\mathcal{B})\leq 2^{n}[\omega]_{A_1}\cdot\omega(\mathcal{B}),
\end{equation*}
that is, $\omega(x)\,dx$ is a doubling measure($\omega$ satisfies the doubling property).
  \item[$(2)$] Let $\omega\in A_2$. In general, for any $t>0$ and for any ball $\mathcal{B}$ in $\mathbb R^n$,
\begin{equation*}
\omega(t\mathcal{B})\leq t^{2n}[\omega]_{A_2}\cdot\omega(\mathcal{B}).
\end{equation*}
\end{enumerate}
\end{lem}

We also need the following two auxiliary lemmas about $\mathrm{BMO}(\omega)$ functions.

\begin{lem}\label{BMOp}
Let $f\in \mathrm{BMO}(\omega)$ and $\omega\in A_1$. Then for any $1\leq p<\infty$ and for any ball $\mathcal{B}$ in $\mathbb R^n$, we get
\begin{equation*}
\bigg(\frac{1}{\omega(\mathcal{B})}\int_{\mathcal{B}}\big|f(x)-f_{\mathcal{B}}\big|^p\omega(x)^{1-p}dx\bigg)^{1/p}
\leq C\big\|f\big\|_{\mathrm{BMO}(\omega)}.
\end{equation*}
\end{lem}
As mentioned before, this estimate was obtained by Garc\'{i}a-Cuerva in \cite{garcia1}. We will give a proof of Garc\'{i}a-Cuerva's result in the last section. When $\omega(x)\equiv1$, see \cite{duoand} and \cite{grafakos}.
\begin{lem}\label{wanglemma1}
Let $f\in \mathrm{BMO}(\omega)$ and $\omega\in A_1$. Then the following properties hold.
\begin{enumerate}
  \item[$(1)$] for any ball $\mathcal{B}$ in $\mathbb R^n$,
  \begin{equation*}
  \frac{1}{m(\mathcal{B})}\int_{\mathcal{B}}\big|f(x)-f_{\mathcal{B}}\big|\,dx
    \leq\Big[[\omega]_{A_1}\underset{x\in \mathcal{B}}{\mathrm{ess\,inf}}\,\omega(x)\Big]\cdot\big\|f\big\|_{\mathrm{BMO}(\omega)}.
  \end{equation*}
  \item[$(2)$] For every $k\in \mathbb{N}$,  we get
\begin{equation*}
\frac{1}{m(2^k\mathcal{B})}\int_{2^k\mathcal{B}}\big|f(x)-f_{\mathcal{B}}\big|\,dx
\leq Ck\Big[[\omega]_{A_1}\underset{x\in \mathcal{B}}{\mathrm{ess\,inf}}\,\omega(x)\Big]\cdot\big\|f\big\|_{\mathrm{BMO}(\omega)}.
\end{equation*}
Here the constant $C$ is independent of $k$ and $f$.
\end{enumerate}
\end{lem}

\begin{proof}
(1). Since $\omega\in A_1$, by \eqref{A1property} and the definition of $\mathrm{BMO}(\omega)$, we have
\begin{equation*}
\begin{split}
\frac{1}{m(\mathcal{B})}\int_{\mathcal{B}}\big|f(x)-f_{\mathcal{B}}\big|\,dx
&=\frac{m(\mathcal{B})}{\omega(\mathcal{B})}\cdot\frac{1}{\omega(\mathcal{B})}\int_{\mathcal{B}}\big|f(x)-f_{\mathcal{B}}\big|\,dx\\
&\leq\Big[[\omega]_{A_1}\underset{x\in \mathcal{B}}{\mathrm{ess\,inf}}\,\omega(x)\Big]\cdot\big\|f\big\|_{\mathrm{BMO}(\omega)},
\end{split}
\end{equation*}
as desired.

(2). Observe that for any ball $\mathcal{B}$ in $\mathbb R^n$,
\begin{equation*}
\begin{split}
\big|f_{2\mathcal{B}}-f_{\mathcal{B}}\big|
&\leq\frac{1}{m(\mathcal{B})}\int_{\mathcal{B}}\big|f(x)-f_{2\mathcal{B}}\big|\,dx\\
&\leq\frac{2^n}{m(2\mathcal{B})}\int_{2\mathcal{B}}\big|f(x)-f_{2\mathcal{B}}\big|\,dx
\leq 2^n\Big[[\omega]_{A_1}\underset{x\in 2\mathcal{B}}{\mathrm{ess\,inf}}\,\omega(x)\Big]\cdot\big\|f\big\|_{\mathrm{BMO}(\omega)},
\end{split}
\end{equation*}
by part (1). Similarly, for each $1\leq i\leq k$,
\begin{equation*}
\big|f_{2^i\mathcal{B}}-f_{2^{i-1}\mathcal{B}}\big|
\leq 2^n\Big[[\omega]_{A_1}\underset{x\in 2^i\mathcal{B}}{\mathrm{ess\,inf}}\,\omega(x)\Big]\cdot\big\|f\big\|_{\mathrm{BMO}(\omega)}.
\end{equation*}
Hence,
\begin{equation*}
\begin{split}
&\frac{1}{m(2^k\mathcal{B})}\int_{2^k\mathcal{B}}\big|f(x)-f_{\mathcal{B}}\big|\,dx
\leq\frac{1}{m(2^k\mathcal{B})}\int_{2^k\mathcal{B}}\big|f(x)-f_{2^k\mathcal{B}}\big|\,dx
+\big|f_{2^k\mathcal{B}}-f_{\mathcal{B}}\big|\\
&\leq\Big[[\omega]_{A_1}\underset{x\in2^k\mathcal{B}}{\mathrm{ess\,inf}}\,\omega(x)\Big]\cdot\big\|f\big\|_{\mathrm{BMO}(\omega)}
+\sum_{i=1}^k\big|f_{2^i\mathcal{B}}-f_{2^{i-1}\mathcal{B}}\big|\\
&\leq\Big[[\omega]_{A_1}\underset{x\in2^k\mathcal{B}}{\mathrm{ess\,inf}}\,\omega(x)\Big]\cdot\big\|f\big\|_{\mathrm{BMO}(\omega)}
+2^n\sum_{i=1}^k\Big[[\omega]_{A_1}\underset{x\in 2^i\mathcal{B}}{\mathrm{ess\,inf}}\,\omega(x)\Big]\cdot\big\|f\big\|_{\mathrm{BMO}(\omega)}\\
&\leq Ck\Big[[\omega]_{A_1}\underset{x\in \mathcal{B}}{\mathrm{ess\,inf}}\,\omega(x)\Big]\cdot\big\|f\big\|_{\mathrm{BMO}(\omega)}.
\end{split}
\end{equation*}
We are done.
\end{proof}

\section{Proofs of Theorems \ref{mainthm1} and \ref{mainthm2}}
Let $\mathcal{F}$ be a real-valued nonnegative function and measurable on $\mathbb R^n$. For each fixed ball $\mathcal{B}\subset\mathbb R^n$, we need the following estimate about the relationship between essential supremum and essential infimum.
\begin{equation}\label{essinf}
\mathcal{F}(x)-\underset{y\in\mathcal{B}}{\mathrm{ess\,inf}}\,\mathcal{F}(y)
\leq \underset{y\in\mathcal{B}}{\mathrm{ess\,sup}}\big|\mathcal{F}(x)-\mathcal{F}(y)\big|.
\end{equation}
We also need the following result. Let $\omega$ be a weight function on $\mathbb R^n$. Since for each ball $\mathcal{B}\subset\mathbb R^n$,
\begin{equation*}
\omega(\mathcal{B})=\int_{\mathcal{B}}\omega(x)\,dx\geq \underset{x\in\mathcal{B}}{\mathrm{ess\,inf}}\,\omega(x)\cdot m(B),
\end{equation*}
then we have
\begin{equation}\label{omegawh}
\frac{m(\mathcal{B})}{\omega(\mathcal{B})}\cdot\underset{x\in\mathcal{B}}{\mathrm{ess\,inf}}\,\omega(x)\leq1.
\end{equation}
Moreover, we will use the following simple fact.
\begin{equation}\label{simfact}
\omega\in A_2\Longleftrightarrow \omega^{-1}\in A_2.
\end{equation}
We are now in a position to prove our main theorems. We first remark that the a.e. existence of the classical and generalized Littlewood--Paley operators was already proved in \cite{sun,zhang1,zhang2} (see also \cite[Chapter 4]{xiao}), under the assumption of one point finiteness. For the multilinear case, one can see \cite{he}. By using similar arguments, we can also obtain the corresponding (a.e. existence) results for the weighted case.

\begin{proof}[Proof of Theorem $\ref{mainthm1}$]
Let ${f}\in \mathrm{BMO}(\omega)$ and $\omega\in A_1$. By the definition of $\mathrm{BLO}(\omega)$, it suffices to show that for any given ball $\mathcal{B}=B(x_0,r)\subset\mathbb R^n$ with center $x_0\in \mathbb R^n$ and radius $r\in(0,+\infty)$, the following inequality holds:
\begin{equation}\label{mainesti1}
\frac{1}{\omega(\mathcal{B})}\int_{\mathcal{B}}
\Big[\big[\mathcal{G}({f})(x)\big]-\underset{y\in\mathcal{B}}{\mathrm{ess\,inf}}\,\big[\mathcal{G}({f})(y)\big]\Big]\,dx
\lesssim\big\|f\big\|_{\mathrm{BMO}(\omega)}.
\end{equation}
First of all, we decompose the integral defining $\mathcal{G}({f})$ into two parts.
\begin{equation*}
\begin{split}
&\big[\mathcal{G}({f})(x)\big]=\bigg(\int_0^{\infty}\big|\psi_t*f(x)\big|^2\frac{dt}{t}\bigg)^{1/2}\\
&=\bigg(\int_0^{r}\big|\psi_t*f(x)\big|^2\frac{dt}{t}+\int_{r}^{\infty}\big|\psi_t*f(x)\big|^2\frac{dt}{t}\bigg)^{1/2}\\
&\leq\bigg(\int_0^{r}\big|\psi_t*f(x)\big|^2\frac{dt}{t}\bigg)^{1/2}
+\bigg(\int_{r}^{\infty}\big|\psi_t*f(x)\big|^2\frac{dt}{t}\bigg)^{1/2}
:=\big[\mathcal{G}_0({f})(x)\big]+\big[\mathcal{G}_{\infty}({f})(x)\big].
\end{split}
\end{equation*}
Obviously,
\begin{equation*}
\big[\mathcal{G}_0({f})(x)\big],\big[\mathcal{G}_{\infty}({f})(x)\big]
\leq \big[\mathcal{G}({f})(x)\big]\leq\big[\mathcal{G}_0({f})(x)\big]+\big[\mathcal{G}_{\infty}({f})(x)\big].
\end{equation*}
Consequently, in view of \eqref{essinf}, we can deduce that
\begin{equation*}
\begin{split}
&\frac{1}{\omega(\mathcal{B})}\int_{\mathcal{B}}
\Big[\big[\mathcal{G}({f})(x)\big]-\underset{y\in\mathcal{B}}{\mathrm{ess\,inf}}\,\big[\mathcal{G}({f})(y)\big]\Big]\,dx\\
&\leq\frac{1}{\omega(\mathcal{B})}\int_{\mathcal{B}}\Big[\big[\mathcal{G}_0({f})(x)\big]+\big[\mathcal{G}_{\infty}({f})(x)\big]
-\underset{y\in\mathcal{B}}{\mathrm{ess\,inf}}\,\big[\mathcal{G}({f})(y)\big]\Big]\,dx\\
&\leq\frac{1}{\omega(\mathcal{B})}\int_{\mathcal{B}}\Big[\big[\mathcal{G}_0({f})(x)\big]+\big[\mathcal{G}_{\infty}({f})(x)\big]
-\underset{y\in\mathcal{B}}{\mathrm{ess\,inf}}\,\big[\mathcal{G}_{\infty}({f})(y)\big]\Big]\,dx\\
&\leq\frac{1}{\omega(\mathcal{B})}\int_{\mathcal{B}}\big[\mathcal{G}_0({f})(x)\big]dx
+\frac{1}{\omega(\mathcal{B})}\int_{\mathcal{B}}\underset{y\in\mathcal{B}}{\mathrm{ess\,sup}}
\Big|\big[\mathcal{G}_{\infty}({f})(x)\big]-\big[\mathcal{G}_{\infty}({f})(y)\big]\Big|\,dx\\
&:=I_0+I_{\infty}.
\end{split}
\end{equation*}
Let us first estimate the term $I_0$. To do this, we decompose the function $f(x)$ as follows.
\begin{equation*}
\begin{split}
f(x)&=f_{2\mathcal{B}}+[f(x)-f_{2\mathcal{B}}]\cdot\chi_{2\mathcal{B}}(x)
+[f(x)-f_{2\mathcal{B}}]\cdot\chi_{(2\mathcal{B})^{\complement}}(x)\\
&:=f^1(x)+f^2(x)+f^3(x),
\end{split}
\end{equation*}
where $2\mathcal{B}=B(x_0,2r)$, $(2\mathcal{B})^{\complement}=\mathbb R^n\setminus(2\mathcal{B})$ and $\chi_{E}$ denotes the
characteristic function of the set $E$. By the vanishing condition of the kernel $\mathcal{\psi}$, we can see that for any $x\in\mathbb R^n$,
\begin{equation*}
\psi_t*f_1(x)\equiv0\Longrightarrow \big[\mathcal{G}_0({f_1})(x)\big]\equiv0.
\end{equation*}
Then we can write
\begin{equation*}
\begin{split}
I_0&=\frac{1}{\omega(\mathcal{B})}\int_{\mathcal{B}}\big[\mathcal{G}_0({f})(x)\big]dx\\
&\leq\frac{1}{\omega(\mathcal{B})}\int_{\mathcal{B}}\big[\mathcal{G}_0({f}^2)(x)\big]dx
+\frac{1}{\omega(\mathcal{B})}\int_{\mathcal{B}}\big[\mathcal{G}_0({f}^3)(x)\big]dx\\
&:=I^2_0+I^3_0.
\end{split}
\end{equation*}
Note that $\omega\in A_1\subset A_2$. Applying Theorem \ref{thmgfunction} with $p=2$ and Cauchy--Schwarz's inequality together with \eqref{simfact}, we obtain
\begin{equation*}
\begin{split}
I^2_0&=\frac{1}{\omega(\mathcal{B})}\int_{\mathcal{B}}\big[\mathcal{G}_0({f}^2)(x)\big]
\omega(x)^{-1/2}\cdot \omega(x)^{1/2}dx\\
&\leq\frac{1}{\omega(\mathcal{B})}\bigg(\int_{\mathcal{B}}\big[\mathcal{G}_0({f}^2)(x)\big]^2\omega(x)^{-1}dx\bigg)^{1/2}
\bigg(\int_{\mathcal{B}}\omega(x)\,dx\bigg)^{1/2}\\
&\leq\frac{1}{\omega(\mathcal{B})^{1/2}}\big\|\mathcal{G}({f}^2)\big\|_{L^2(\omega^{-1})}\\
&\leq C\cdot\frac{\omega(2\mathcal{B})^{1/2}}{\omega(\mathcal{B})^{1/2}}
\bigg(\frac{1}{\omega(2\mathcal{B})}\int_{2\mathcal{B}}\big|f(x)-f_{2\mathcal{B}}\big|^2\omega(x)^{-1}dx\bigg)^{1/2}.
\end{split}
\end{equation*}
Moreover, by Lemma \ref{BMOp} with $p=2$ and part (1) of Lemma \ref{lemmaA1}, we have
\begin{equation*}
\begin{split}
I^2_0&\leq C\cdot\frac{\omega(2\mathcal{B})^{1/2}}{\omega(\mathcal{B})^{1/2}}\big\|f\big\|_{\mathrm{BMO}(\omega)}
\lesssim\big\|f\big\|_{\mathrm{BMO}(\omega)}.
\end{split}
\end{equation*}
Note that for any $x\in\mathcal{B}=B(x_0,r)$ and $z\in\mathbb R^n\setminus(2\mathcal{B})$, one has
\begin{equation*}
|x-z|\leq|x-x_0|+|x_0-z|<r+|x_0-z|\leq\frac{3}{\,2\,}|x_0-z|,
\end{equation*}
and
\begin{equation*}
|x_0-z|\leq|x_0-x|+|x-z|<r+|x-z|\leq\frac{1}{\,2\,}|x_0-z|+|x-z|.
\end{equation*}
That is, $|x-z|\approx|x_0-z|$. Thus, by the size condition of the kernel $\psi$, we can see that for any $x\in\mathcal{B}$,
\begin{equation*}
\begin{split}
&\big|\psi_t*f^3(x)\big|
=\bigg|\int_{\mathbb R^n\setminus 2\mathcal{B}}\psi_t(x-z)\big[f(z)-f_{2\mathcal{B}}\big]\,dz\bigg|\\
&\lesssim\int_{\mathbb R^n\setminus 2\mathcal{B}}\frac{t^{\delta}}{(t+|x-z|)^{n+\delta}}
\big|f(z)-f_{2\mathcal{B}}\big|\,dz\\
&\lesssim t^{\delta}\cdot\sum_{j=1}^{\infty}\int_{2^{j+1}\mathcal{B}\setminus2^j \mathcal{B}}
\frac{1}{|x_0-z|^{n+\delta}}\big|f(z)-f_{2\mathcal{B}}\big|\,dz\\
&\lesssim t^{\delta}\cdot\sum_{j=1}^{\infty}\int_{2^{j+1}\mathcal{B}}
\frac{1}{m(2^{j+1}\mathcal{B})^{1+\delta/n}}\big|f(z)-f_{2\mathcal{B}}\big|\,dz.
\end{split}
\end{equation*}
It then follows from part (2) of Lemma \ref{wanglemma1} that
\begin{equation*}
\begin{split}
\big|\psi_t*f^3(x)\big|&\lesssim t^{\delta}\cdot\sum_{j=1}^{\infty}
\bigg\{\frac{1}{m(2^j\mathcal{B})^{\frac{\delta}{n}}}
\Big[j\cdot\big\|f\big\|_{\mathrm{BMO}(\omega)}\underset{z\in 2\mathcal{B}}{\mathrm{ess\,inf}}\,\omega(z)\Big]\bigg\}\\
&\lesssim t^{\delta}\cdot\bigg\{\sum_{j=1}^{\infty}\frac{j}{m(2^j\mathcal{B})^{\frac{\delta}{n}}}\bigg\}
\times\big\|f\big\|_{\mathrm{BMO}(\omega)}\underset{z\in\mathcal{B}}{\mathrm{ess\,inf}}\,\omega(z).
\end{split}
\end{equation*}
Therefore,
\begin{equation*}
\begin{split}
I^3_0&=\frac{1}{\omega(\mathcal{B})}\int_{\mathcal{B}}\big[\mathcal{G}_0({f}^3)(x)\big]dx\\
&\lesssim\frac{1}{\omega(\mathcal{B})}\int_{\mathcal{B}}\bigg(\int_0^{r}t^{2\delta-1}dt\bigg)^{1/2}
\bigg\{\sum_{j=1}^{\infty}\frac{j}{m(2^j\mathcal{B})^{\frac{\delta}{n}}}\bigg\}
\times\big\|f\big\|_{\mathrm{BMO}(\omega)}\underset{x\in\mathcal{B}}{\mathrm{ess\,inf}}\,\omega(x)\\
&\lesssim r^{\delta}\bigg\{\sum_{j=1}^{\infty}\frac{j}{(2^j)^{\delta}r^{\delta}}\bigg\}
\times\big\|f\big\|_{\mathrm{BMO}(\omega)}
\bigg\{\frac{m(\mathcal{B})}{\omega(\mathcal{B})}\cdot\underset{x\in\mathcal{B}}{\mathrm{ess\,inf}}\,\omega(x)\bigg\}\\
&\lesssim\big\|f\big\|_{\mathrm{BMO}(\omega)},
\end{split}
\end{equation*}
where in the last inequality we have used \eqref{omegawh}. On the other hand, by the smoothness condition of the kernel $\psi$, we can see that when $x,y\in \mathcal{B}=B(x_0,r)$ and $z\in \mathbb R^n\setminus(4\mathcal{B})$,
\begin{equation}\label{regularity}
\begin{split}
\Big|\psi_t(x-z)-\psi_t(y-z)\Big|
&=\frac{1}{t^{n}}\cdot\bigg|\psi\Big(\frac{x-z}{t}\Big)-\psi\Big(\frac{y-z}{t}\Big)\bigg|\\
&\lesssim\frac{t^{\delta}\cdot|x-y|^{\gamma}}{(t+|x-z|)^{n+\delta+\gamma}}.
\end{split}
\end{equation}
Consequently, by \eqref{regularity}, the triangle inequality and the vanishing condition of the kernel $\psi$, we obtain that for any $x,y\in\mathcal{B}=B(x_0,r)$,
\begin{equation}\label{firstsecond2}
\begin{split}
&\Big|\big|\psi_t*f(x)\big|-\big|\psi_t*f(y)\big|\Big|\leq\Big|\psi_t*f(x)-\psi_t*f(y)\Big|\\
&=\bigg|\int_{\mathbb R^n}\big[\psi_t(x-z)-\psi_t(y-z)\big]\cdot\big[f(z)-f_{4\mathcal{B}}\big]\,dz\bigg|\\
&\lesssim\int_{4\mathcal{B}}\big[\big|\psi_t(x-z)\big|+\big|\psi_t(y-z)\big|\big]\big|f(z)-f_{4\mathcal{B}}\big|\,dz\\
&+\int_{\mathbb R^n\setminus4\mathcal{B}}\frac{t^{\delta}\cdot|x-y|^{\gamma}}{(t+|x-z|)^{n+\delta+\gamma}}
\big|f(z)-f_{4\mathcal{B}}\big|\,dz.
\end{split}
\end{equation}
By the size condition of the kernel $\psi$, we know that for any $t>0$, $x,z\in \mathbb R^n$,
\begin{equation}\label{A83}
\big|\psi_t(x-z)\big|\lesssim\frac{1}{t^{n}}.
\end{equation}
In view of \eqref{A83}, the first term in \eqref{firstsecond2} is naturally controlled by
\begin{equation*}
\begin{split}
\frac{1}{t^{n}}\cdot\int_{4\mathcal{B}}\big|f(z)-f_{4\mathcal{B}}\big|\,dz.
\end{split}
\end{equation*}
We now proceed to estimate the second term in \eqref{firstsecond2}. Observe that when $t>r$, $x\in\mathcal{B}=B(x_0,r)$ and $z\in\mathbb R^n\setminus(4\mathcal{B})$,
\begin{equation*}
t+|x-z|\leq t+|x-x_0|+|x_0-z|<2t+|x_0-z|\leq 2(t+|x_0-z|),
\end{equation*}
and
\begin{equation*}
t+|x_0-z|\leq t+|x_0-x|+|x-z|<2t+|x-z|\leq 2(t+|x-z|).
\end{equation*}
That is, $t+|x-z|\approx t+|x_0-z|$. Thus, the second term in \eqref{firstsecond2} is bounded by
\begin{equation*}
t^{\delta}\cdot\int_{\mathbb R^n\setminus4\mathcal{B}}\frac{(2r)^{\gamma}}{(t+|x_0-z|)^{n+\delta+\gamma}}
\big|f(z)-f_{4\mathcal{B}}\big|\,dz.
\end{equation*}
Hence, by Minkowski's inequality, we can deduce that for any $x,y\in \mathcal{B}=B(x_0,r)$,
\begin{equation*}
\begin{split}
&\Big|\big[\mathcal{G}_{\infty}({f})(x)\big]-\big[\mathcal{G}_{\infty}({f})(y)\big]\Big|\\
&\leq\bigg(\int_{r}^{\infty}\Big|\big|\psi_t*f(x)\big|-\big|\psi_t*f(y)\big|\Big|^2\frac{dt}{t}\bigg)^{1/2}\\
&\lesssim\bigg(\int_{r}^{\infty}\bigg\{\frac{1}{t^{n}}\cdot\int_{4\mathcal{B}}\big|f(z)-f_{4\mathcal{B}}\big|\,dz\bigg\}^2\frac{dt}{t}\bigg)^{1/2}\\
&+\bigg(\int_{r}^{\infty}\bigg\{t^{\delta}\cdot\int_{\mathbb R^n\setminus4\mathcal{B}}\frac{(2r)^{\gamma}}{(t+|x_0-z|)^{n+\delta+\gamma}}
\big|f(z)-f_{4\mathcal{B}}\big|\,dz\bigg\}^2\frac{dt}{t}\bigg)^{1/2}\\
&\leq \int_{4\mathcal{B}}\big|f(z)-f_{4\mathcal{B}}\big|\,dz
\bigg\{\int_{r}^{\infty}\frac{1}{t^{2n+1}}\,dt\bigg\}^{1/2}\\
&+\int_{\mathbb R^n\setminus4\mathcal{B}}(2r)^{\gamma}\big|f(z)-f_{4\mathcal{B}}\big|\,dz
\bigg\{\int_{r}^{\infty}\frac{t^{2\delta-1}}{(t+|x_0-z|)^{2(n+\delta+\gamma)}}\,dt\bigg\}^{1/2}.
\end{split}
\end{equation*}
For $\delta,\gamma>0$, a direct computation shows that
\begin{equation*}
\bigg\{\int_{r}^{\infty}\frac{1}{t^{2n+1}}\,dt\bigg\}^{1/2}\lesssim\frac{1}{r^n},
\end{equation*}
and
\begin{equation*}
\begin{split}
&\bigg\{\int_{r}^{\infty}\frac{t^{2\delta-1}}{(t+|x_0-z|)^{2(n+\delta+\gamma)}}\,dt\bigg\}^{1/2}\\
&=\bigg\{\int_{r}^{|x_0-z|}\frac{t^{2\delta-1}}{(t+|x_0-z|)^{2(n+\delta+\gamma)}}\,dt
+\int_{|x_0-z|}^{\infty}\frac{t^{2\delta-1}}{(t+|x_0-z|)^{2(n+\delta+\gamma)}}\,dt\bigg\}^{1/2}\\
&\leq\bigg\{\int_{r}^{|x_0-z|}\frac{t^{2\delta-1}}{(|x_0-z|)^{2(n+\delta+\gamma)}}\,dt\bigg\}^{1/2}+
\bigg\{\int_{|x_0-z|}^{\infty}\frac{t^{2\delta-1}}{t^{2(n+\delta+\gamma)}}\,dt\bigg\}^{1/2}\\
&\leq\bigg\{\int_0^{|x_0-z|}\frac{t^{2\delta-1}}{(|x_0-z|)^{2(n+\delta+\gamma)}}\,dt\bigg\}^{1/2}+
\bigg\{\int_{|x_0-z|}^{\infty}\frac{1}{t^{2(n+\gamma)+1}}\,dt\bigg\}^{1/2}\\
&\lesssim\frac{1}{(|x_0-z|)^{n+\gamma}}.
\end{split}
\end{equation*}
From this, it follows that
\begin{equation*}
\begin{split}
&\Big|\big[\mathcal{G}_{\infty}({f})(x)\big]-\big[\mathcal{G}_{\infty}({f})(y)\big]\Big|
\lesssim\frac{1}{m(4\mathcal{B})}\int_{4\mathcal{B}}\big|f(z)-f_{4\mathcal{B}}\big|\,dz\\
&+\sum_{j=2}^{\infty}\int_{2^{j+1}\mathcal{B}\setminus2^j\mathcal{B}}
\frac{(2r)^{\gamma}}{(|x_0-z|)^{n+\gamma}}\big|f(z)-f_{4\mathcal{B}}\big|\,dz\\
&\lesssim\frac{1}{m(4\mathcal{B})}\int_{4\mathcal{B}}\big|f(z)-f_{4\mathcal{B}}\big|\,dz
+\sum_{j=2}^{\infty}\bigg\{\frac{(2r)^{\gamma}}{(2^jr)^{\gamma}}\cdot\frac{1}{m(2^{j+1}\mathcal{B})}
\int_{2^{j+1} \mathcal{B}}\big|f(z)-f_{4\mathcal{B}}\big|\,dz\bigg\}.
\end{split}
\end{equation*}
Furthermore, from part (1) and part (2) of Lemma \ref{wanglemma1}, it then follows that for any $x,y\in \mathcal{B}=B(x_0,r)$,
\begin{equation*}
\begin{split}
&\Big|\big[\mathcal{G}_{\infty}({f})(x)\big]-\big[\mathcal{G}_{\infty}({f})(y)\big]\Big|\\
&\lesssim\big\|f\big\|_{\mathrm{BMO}(\omega)}\underset{z\in 4\mathcal{B}}{\mathrm{ess\,inf}}\,\omega(z)
+\sum_{j=2}^{\infty}\frac{j}{2^{(j-1)\gamma}}
\times\big\|f\big\|_{\mathrm{BMO}(\omega)}\underset{z\in4\mathcal{B}}{\mathrm{ess\,inf}}\,\omega(z)\\
&\lesssim\big\|f\big\|_{\mathrm{BMO}(\omega)}\underset{z\in4\mathcal{B}}{\mathrm{ess\,inf}}\,\omega(z).
\end{split}
\end{equation*}
Therefore,
\begin{equation*}
\begin{split}
I_{\infty}&=\frac{1}{\omega(\mathcal{B})}\int_{\mathcal{B}}\underset{y\in\mathcal{B}}{\mathrm{ess\,sup}}
\Big|\big[\mathcal{G}_{\infty}({f})(x)\big]-\big[\mathcal{G}_{\infty}({f})(y)\big]\Big|\,dx\\
&\lesssim\big\|f\big\|_{\mathrm{BMO}(\omega)}
\bigg\{\frac{m(\mathcal{B})}{\omega(\mathcal{B})}\cdot\underset{z\in4\mathcal{B}}{\mathrm{ess\,inf}}\,\omega(z)\bigg\}\\
&\leq\big\|f\big\|_{\mathrm{BMO}(\omega)}
\bigg\{\frac{m(\mathcal{B})}{\omega(\mathcal{B})}\cdot\underset{z\in\mathcal{B}}{\mathrm{ess\,inf}}\,\omega(z)\bigg\}
\leq\big\|f\big\|_{\mathrm{BMO}(\omega)},
\end{split}
\end{equation*}
where the last inequality follows from \eqref{omegawh}. Combining the above estimates for both terms $I_0$ and $I_{\infty}$ yields the desired result \eqref{mainesti1}. This completes the proof of Theorem \ref{mainthm1}.
\end{proof}

\begin{proof}[Proof of Theorem $\ref{mainthm2}$]
Let ${f}\in \mathrm{BMO}(\omega)$ and $\omega\in A_1$. By the definition of $\mathrm{BLO}(\omega)$, it suffices to prove that for any given ball $\mathcal{B}=B(x_0,r)\subset\mathbb R^n$ centered at $x_0\in \mathbb R^n$ with radius $r\in(0,+\infty)$, the following inequality holds:
\begin{equation}\label{mainesti2}
\frac{1}{\omega(\mathcal{B})}\int_{\mathcal{B}}
\Big[\big[\mathcal{S}({f})(x)\big]-\underset{y\in\mathcal{B}}{\mathrm{ess\,inf}}\,\big[\mathcal{S}({f})(y)\big]\Big]\,dx
\lesssim\big\|f\big\|_{\mathrm{BMO}(\omega)}.
\end{equation}
To prove \eqref{mainesti2}, we first decompose the integral defining $\mathcal{S}({f})$ into two parts.
\begin{equation*}
\begin{split}
\big[\mathcal{S}({f})(x)\big]
&=\bigg(\int_0^{\infty}\int_{|z-x|<t}\big|\psi_t*f(z)\big|^2\frac{dzdt}{t^{n+1}}\bigg)^{1/2}\\
&=\bigg(\int_0^{r}\int_{|z-x|<t}\big|\psi_t*f(z)\big|^2\frac{dzdt}{t^{n+1}}
+\int_{r}^{\infty}\int_{|z-x|<t}\big|\psi_t*f(z)\big|^2\frac{dzdt}{t^{n+1}}\bigg)^{1/2}\\
\end{split}
\end{equation*}
\begin{equation*}
\begin{split}
&\leq\bigg(\int_0^{r}\int_{|z-x|<t}\big|\psi_t*f(z)\big|^2\frac{dzdt}{t^{n+1}}\bigg)^{1/2}
+\bigg(\int_{r}^{\infty}\int_{|z-x|<t}\big|\psi_t*f(z)\big|^2\frac{dzdt}{t^{n+1}}\bigg)^{1/2}\\
&:=\big[\mathcal{S}_0({f})(x)\big]+\big[\mathcal{S}_{\infty}({f})(x)\big].
\end{split}
\end{equation*}
It is clear that
\begin{equation*}
\big[\mathcal{S}_0({f})(x)\big],\big[\mathcal{S}_{\infty}({f})(x)\big]
\leq \big[\mathcal{S}({f})(x)\big]\leq\big[\mathcal{S}_0({f})(x)\big]+\big[\mathcal{S}_{\infty}({f})(x)\big].
\end{equation*}
Consequently, in view of \eqref{essinf}, we can deduce that
\begin{equation*}
\begin{split}
&\frac{1}{\omega(\mathcal{B})}\int_{\mathcal{B}}
\Big[\big[\mathcal{S}({f})(x)\big]-\underset{y\in\mathcal{B}}{\mathrm{ess\,inf}}\,\big[\mathcal{S}({f})(y)\big]\Big]\,dx\\
&\leq\frac{1}{\omega(\mathcal{B})}\int_{\mathcal{B}}
\Big[\big[\mathcal{S}_0({f})(x)\big]+\big[\mathcal{S}_{\infty}({f})(x)\big]
-\underset{y\in\mathcal{B}}{\mathrm{ess\,inf}}\,\big[\mathcal{S}(\vec{f})(y)\big]\Big]\,dx\\
&\leq\frac{1}{\omega(\mathcal{B})}\int_{\mathcal{B}}
\Big[\big[\mathcal{S}_0({f})(x)\big]+\big[\mathcal{S}_{\infty}({f})(x)\big]
-\underset{y\in\mathcal{B}}{\mathrm{ess\,inf}}\,\big[\mathcal{S}_{\infty}({f})(y)\big]\Big]\,dx\\
&\leq\frac{1}{\omega(\mathcal{B})}\int_{\mathcal{B}}\big[\mathcal{S}_0({f})(x)\big]dx
+\frac{1}{\omega(\mathcal{B})}\int_{\mathcal{B}}
\underset{y\in\mathcal{B}}{\mathrm{ess\,sup}}\Big|\big[\mathcal{S}_{\infty}({f})(x)\big]-\big[\mathcal{S}_{\infty}({f})(y)\big]\Big|\,dx\\
&:=J_0+J_{\infty}.
\end{split}
\end{equation*}
As usual, we decompose the function $f(x)$ in the following way
\begin{equation*}
\begin{split}
f(x)&=f_{4\mathcal{B}}+[f(x)-f_{4\mathcal{B}}]\cdot\chi_{4\mathcal{B}}(x)
+[f(x)-f_{4\mathcal{B}}]\cdot\chi_{(4\mathcal{B})^{\complement}}(x)\\
&:=f^1(x)+f^2(x)+f^3(x),
\end{split}
\end{equation*}
where $4\mathcal{B}=B(x_0,4r)$ and $(4\mathcal{B})^{\complement}=\mathbb R^n\setminus(4\mathcal{B})$. By the vanishing condition of the kernel ${\psi}$, we can see that
\begin{equation*}
\psi_t*f(x)=\psi_t*f^2(x)+\psi_t*f^3(x)\quad \& \quad \big[\mathcal{S}_0({f}^1)(x)\big]\equiv0.
\end{equation*}
Then we can write
\begin{equation*}
\begin{split}
J_0&=\frac{1}{\omega(\mathcal{B})}\int_{\mathcal{B}}\big[\mathcal{S}_0({f})(x)\big]dx\\
&\leq\frac{1}{\omega(\mathcal{B})}\int_{\mathcal{B}}\big[\mathcal{S}_0({f}^2)(x)\big]dx
+\frac{1}{\omega(\mathcal{B})}\int_{\mathcal{B}}\big[\mathcal{S}_0({f}^3)(x)\big]dx\\
&:=J^2_0+J^3_0.
\end{split}
\end{equation*}
Notice that $\omega\in A_1\subset A_2$. Applying Theorem \ref{thmsfunction} with $p=2$, Cauchy--Schwarz's inequality and \eqref{simfact}, we obtain
\begin{equation*}
\begin{split}
J^2_0&=\frac{1}{\omega(\mathcal{B})}\int_{\mathcal{B}}\big[\mathcal{S}_0({f}^2)(x)\big]
\omega(x)^{-1/2}\cdot \omega(x)^{1/2}dx\\
&\leq\frac{1}{\omega(\mathcal{B})}\bigg(\int_{\mathcal{B}}\big[\mathcal{S}_0({f}^2)(x)\big]^2\omega(x)^{-1}dx\bigg)^{1/2}
\bigg(\int_{\mathcal{B}}\omega(x)\,dx\bigg)^{1/2}\\
&\leq\frac{1}{\omega(\mathcal{B})^{1/2}}\big\|\mathcal{S}({f}^2)\big\|_{L^2(\omega^{-1})}\\
&\leq C\cdot\frac{\omega(4\mathcal{B})^{1/2}}{\omega(\mathcal{B})^{1/2}}
\bigg(\frac{1}{\omega(4\mathcal{B})}\int_{4\mathcal{B}}\big|f(x)-f_{4\mathcal{B}}\big|^2\omega(x)^{-1}dx\bigg)^{1/2}.
\end{split}
\end{equation*}
Moreover, by using Lemma \ref{BMOp} with $p=2$ and part (1) of Lemma \ref{lemmaA1}, we have
\begin{equation*}
\begin{split}
J^2_0&\leq C\cdot\frac{\omega(4\mathcal{B})^{1/2}}{\omega(\mathcal{B})^{1/2}}\big\|f\big\|_{\mathrm{BMO}(\omega)}
\lesssim\big\|f\big\|_{\mathrm{BMO}(\omega)}.
\end{split}
\end{equation*}
Notice that for any $x\in\mathcal{B}=B(x_0,r)$, $|z-x|<t$ with $t<r$ and $\xi\in\mathbb R^n\setminus(4\mathcal{B})$, one has
\begin{equation*}
\begin{split}
|z-\xi|&\leq |z-x|+|x-x_0|+|x_0-\xi|\\
&<t+r+|x_0-\xi|<2r+|x_0-\xi|\leq\frac{3}{\,2\,}|x_0-\xi|,
\end{split}
\end{equation*}
and
\begin{equation*}
\begin{split}
|x_0-\xi|&\leq|x_0-x|+|x-z|+|z-\xi|\\
&<r+t+|z-\xi|\leq\frac{1}{\,2\,}|x_0-\xi|+|z-\xi|.
\end{split}
\end{equation*}
That is, $|z-\xi|\approx|x_0-\xi|$. Hence, by the size condition of the kernel $\psi$, we can see that for any $x\in\mathcal{B}$,
\begin{equation*}
\begin{split}
&\big|\psi_t*f^3(z)\big|
=\bigg|\int_{\mathbb R^n\setminus 4\mathcal{B}}\psi_t(z-\xi)\big[f(\xi)-f_{4\mathcal{B}}\big]\,d\xi\bigg|\\
&\lesssim\int_{\mathbb R^n\setminus 4\mathcal{B}}\frac{t^{\delta}}{(t+|z-\xi|)^{n+\delta}}
\big|f(\xi)-f_{4\mathcal{B}}\big|\,d\xi\\
&\lesssim t^{\delta}\cdot\sum_{j=2}^{\infty}\int_{2^{j+1}\mathcal{B}\setminus2^j \mathcal{B}}
\frac{1}{|x_0-\xi|^{n+\delta}}\big|f(\xi)-f_{4\mathcal{B}}\big|\,d\xi\\
&\lesssim t^{\delta}\cdot\sum_{j=2}^{\infty}\int_{2^{j+1}\mathcal{B}}
\frac{1}{m(2^{j+1}\mathcal{B})^{1+\delta/n}}\big|f(\xi)-f_{4\mathcal{B}}\big|\,d\xi.
\end{split}
\end{equation*}
Moreover, by part (2) of Lemma \ref{wanglemma1}, we get
\begin{equation*}
\begin{split}
\big|\psi_t*f^3(z)\big|&\lesssim t^{\delta}\cdot\sum_{j=2}^{\infty}
\bigg\{\frac{1}{m(2^j\mathcal{B})^{\frac{\delta}{n}}}
\Big[j\cdot\big\|f\big\|_{\mathrm{BMO}(\omega)}\underset{\xi\in 4\mathcal{B}}{\mathrm{ess\,inf}}\,\omega(\xi)\Big]\bigg\}\\
&\lesssim t^{\delta}\cdot\bigg\{\sum_{j=2}^{\infty}\frac{j}{m(2^j\mathcal{B})^{\frac{\delta}{n}}}\bigg\}
\times\big\|f\big\|_{\mathrm{BMO}(\omega)}\underset{\xi\in\mathcal{B}}{\mathrm{ess\,inf}}\,\omega(\xi).
\end{split}
\end{equation*}
Therefore,
\begin{equation*}
\begin{split}
J^3_0&=\frac{1}{\omega(\mathcal{B})}\int_{\mathcal{B}}\big[\mathcal{S}_0({f}^3)(x)\big]dx\\
&\lesssim\frac{1}{\omega(\mathcal{B})}\int_{\mathcal{B}}\bigg(\int_0^{r}\int_{|z-x|<t}t^{2\delta}\frac{dzdt}{t^{n+1}}\bigg)^{1/2}
\bigg\{\sum_{j=2}^{\infty}\frac{j}{m(2^j\mathcal{B})^{\frac{\delta}{n}}}\bigg\}\\
&\times\big\|f\big\|_{\mathrm{BMO}(\omega)}\underset{x\in\mathcal{B}}{\mathrm{ess\,inf}}\,\omega(x)\\
\end{split}
\end{equation*}
\begin{equation*}
\begin{split}
&\lesssim\frac{1}{\omega(\mathcal{B})}\int_{\mathcal{B}}\bigg(\int_0^{r}t^{2\delta-1}dt\bigg)^{1/2}
\bigg\{\sum_{j=2}^{\infty}\frac{j}{m(2^j\mathcal{B})^{\frac{\delta}{n}}}\bigg\}
\times\big\|f\big\|_{\mathrm{BMO}(\omega)}\underset{x\in\mathcal{B}}{\mathrm{ess\,inf}}\,\omega(x)\\
&\lesssim r^{\delta}\bigg\{\sum_{j=2}^{\infty}\frac{j}{(2^j)^{\delta}r^{\delta}}\bigg\}
\times\big\|f\big\|_{\mathrm{BMO}(\omega)}
\bigg\{\frac{m(\mathcal{B})}{\omega(\mathcal{B})}\cdot\underset{x\in\mathcal{B}}{\mathrm{ess\,inf}}\,\omega(x)\bigg\}\\
&\lesssim\big\|f\big\|_{\mathrm{BMO}(\omega)},
\end{split}
\end{equation*}
where the last inequality follows from \eqref{omegawh}. On the other hand, by the smoothness condition of the kernel $\psi$, we can see that when $x,y\in \mathcal{B}=B(x_0,r)$, $|z|<t$ with $t>r$ and $\xi\in \mathbb R^n\setminus(8\mathcal{B})$,
\begin{equation}\label{regularity2}
\begin{split}
\Big|\psi_t(x+z-\xi)-\psi_t(y+z-\xi)\Big|
&=\frac{1}{t^{n}}\cdot\bigg|\psi\Big(\frac{x+z-\xi}{t}\Big)-\psi\Big(\frac{y+z-\xi}{t}\Big)\bigg|\\
&\lesssim\frac{t^{\delta}\cdot|x-y|^{\gamma}}{(t+|x+z-\xi|)^{n+\delta+\gamma}}.
\end{split}
\end{equation}
Consequently, by \eqref{regularity2}, the triangle inequality and the vanishing condition of the kernel $\psi$, we obtain that for any $x,y\in\mathcal{B}=B(x_0,r)$,
\begin{equation}\label{firstsecond4}
\begin{split}
&\Big|\big|\psi_t*f(x+z)\big|-\big|\psi_t*f(y+z)\big|\Big|\leq\Big|\psi_t*f(x+z)-\psi_t*f(y+z)\Big|\\
&=\bigg|\int_{\mathbb R^n}\big[\psi_t(x+z-\xi)-\psi_t(y+z-\xi)\big]\cdot\big[f(\xi)-f_{8\mathcal{B}}\big]\,d\xi\bigg|\\
&\lesssim\int_{8\mathcal{B}}\big[\big|\psi_t(x+z-\xi)\big|+\big|\psi_t(y+z-\xi)\big|\big]\big|f(\xi)-f_{8\mathcal{B}}\big|\,d\xi\\
&+\int_{\mathbb R^n\setminus8\mathcal{B}}\frac{t^{\delta}\cdot|x-y|^{\gamma}}{(t+|x+z-\xi|)^{n+\delta+\gamma}}
\big|f(\xi)-f_{8\mathcal{B}}\big|\,d\xi.
\end{split}
\end{equation}
From the previous estimate \eqref{A83}, the first term in \eqref{firstsecond4} is naturally controlled by
\begin{equation*}
\begin{split}
\frac{1}{t^{n}}\cdot\int_{8\mathcal{B}}\big|f(\xi)-f_{8\mathcal{B}}\big|\,d\xi.
\end{split}
\end{equation*}
Finally, we consider the second term in \eqref{firstsecond4}. Observe that for any $|z|<t$ with $t>r$, $x\in\mathcal{B}=B(x_0,r)$ and $\xi\in \mathbb R^n\setminus(8\mathcal{B})$,
\begin{equation*}
\begin{split}
t+|x+z-\xi|&\leq t+|z|+|x-x_0|+|x_0-\xi|\\
&<2t+r+|x_0-\xi|\\
&<3t+|x_0-\xi|\leq3(t+|x_0-\xi|),
\end{split}
\end{equation*}
and
\begin{equation*}
\begin{split}
t+|x_0-\xi|&\leq t+|x_0-x|+|x+z-\xi|+|z|\\
&<r+2t+|x+z-\xi|\\
&<3t+|x+z-\xi|\leq 3(t+|x+z-\xi|).
\end{split}
\end{equation*}
This gives
\begin{equation*}
t+|x+z-\xi|\approx t+|x_0-\xi|.
\end{equation*}
Then the second term in \eqref{firstsecond4} is bounded by
\begin{equation*}
t^{\delta}\cdot\int_{\mathbb R^n\setminus 8\mathcal{B}}
\frac{(2r)^{\gamma}}{(t+|x_0-\xi|)^{n+\delta+\gamma}}\big|f(\xi)-f_{8\mathcal{B}}\big|\,d\xi.
\end{equation*}
Thus, by using Minkowski's inequality, we can deduce that for any $x,y\in \mathcal{B}=B(x_0,r)$,
\begin{equation*}
\begin{split}
&\Big|\big[\mathcal{S}_{\infty}({f})(x)\big]-\big[\mathcal{S}_{\infty}({f})(y)\big]\Big|\\
&=\bigg|\bigg(\int_{r}^{\infty}\int_{|z|<t}\big|\psi_t*f(x+z)\big|^2\frac{dzdt}{t^{n+1}}\bigg)^{1/2}
-\bigg(\int_{r}^{\infty}\int_{|z|<t}\big|\psi_t*f(y+z)\big|^2\frac{dzdt}{t^{n+1}}\bigg)^{1/2}\bigg|\\
&\leq\bigg(\int_{r}^{\infty}\int_{|z|<t}
\Big|\big|\psi_t*f(x+z)\big|-\big|\psi_t*f(y+z)\big|\Big|^2\frac{dzdt}{t^{n+1}}\bigg)^{1/2}\\
&\lesssim\bigg(\int_{r}^{\infty}\int_{|z|<t}
\bigg\{\frac{1}{t^{n}}\cdot\int_{8\mathcal{B}}\big|f(\xi)-f_{8\mathcal{B}}\big|\,d\xi\bigg\}^2
\frac{dzdt}{t^{n+1}}\bigg)^{1/2}\\
&+\bigg(\int_{r}^{\infty}\int_{|z|<t}
\bigg\{t^{\delta}\cdot\int_{\mathbb R^n\setminus8\mathcal{B}}\frac{(2r)^{\gamma}}{(t+|x_0-\xi|)^{n+\delta+\gamma}}
\big|f(\xi)-f_{8\mathcal{B}}\big|\,d\xi\bigg\}^2
\frac{dzdt}{t^{n+1}}\bigg)^{1/2}\\
&\leq \int_{8\mathcal{B}}\big|f(\xi)-f_{8\mathcal{B}}\big|\,d\xi
\bigg\{\int_{r}^{\infty}\int_{|z|<t}\frac{1}{t^{3n+1}}\,dzdt\bigg\}^{1/2}\\
&+\int_{\mathbb R^n\setminus8\mathcal{B}}(2r)^{\gamma}\big|f(\xi)-f_{8\mathcal{B}}\big|\,d\xi
\bigg\{\int_{r}^{\infty}\int_{|z|<t}\frac{t^{2\delta-n-1}}{(t+|x_0-\xi|)^{2(n+\delta+\gamma)}}\,dzdt\bigg\}^{1/2}.
\end{split}
\end{equation*}
We now proceed exactly as in Theorem \ref{mainthm1}, and obtain
\begin{equation*}
\begin{split}
&\bigg\{\int_{r}^{\infty}\int_{|z|<t}\frac{1}{t^{3n+1}}\,dzdt\bigg\}^{1/2}\\
&\lesssim\bigg\{\int_{r}^{\infty}\frac{1}{t^{2n+1}}\,dt\bigg\}^{1/2}\lesssim\frac{1}{r^n},
\end{split}
\end{equation*}
and
\begin{equation*}
\begin{split}
&\bigg\{\int_{r}^{\infty}\int_{|z|<t}\frac{t^{2\delta-n-1}}{(t+|x_0-\xi|)^{2(n+\delta+\gamma)}}\,dzdt\bigg\}^{1/2}\\
&\lesssim\bigg\{\int_{r}^{\infty}\frac{t^{2\delta-1}}{(t+|x_0-\xi|)^{2(n+\delta+\gamma)}}\,dt\bigg\}^{1/2}
\lesssim\frac{1}{(|x_0-\xi|)^{n+\gamma}}.
\end{split}
\end{equation*}
Furthermore, from part (1) and part (2) of Lemma \ref{wanglemma1}, it follows that for any $x,y\in \mathcal{B}=B(x_0,r)$,
\begin{equation*}
\begin{split}
&\Big|\big[\mathcal{S}_{\infty}({f})(x)\big]-\big[\mathcal{S}_{\infty}({f})(y)\big]\Big|
\lesssim\frac{1}{m(8\mathcal{B})}\int_{8\mathcal{B}}\big|f(\xi)-f_{8\mathcal{B}}\big|\,d\xi\\
&+\sum_{j=3}^{\infty}\int_{2^{j+1}\mathcal{B}\setminus2^j\mathcal{B}}
\frac{(2r)^{\gamma}}{(|x_0-\xi|)^{n+\gamma}}\big|f(\xi)-f_{8\mathcal{B}}\big|\,d\xi\\
&\lesssim\frac{1}{m(8\mathcal{B})}\int_{8\mathcal{B}}\big|f(\xi)-f_{8\mathcal{B}}\big|\,d\xi
+\sum_{j=3}^{\infty}\bigg\{\frac{(2r)^{\gamma}}{(2^jr)^{\gamma}}\cdot\frac{1}{m(2^{j+1}\mathcal{B})}
\int_{2^{j+1}\mathcal{B}}\big|f(\xi)-f_{8\mathcal{B}}\big|\,d\xi\bigg\}\\
&\lesssim\big\|f\big\|_{\mathrm{BMO}(\omega)}\underset{\xi\in 8\mathcal{B}}{\mathrm{ess\,inf}}\,\omega(\xi)
+\sum_{j=3}^{\infty}\frac{j}{2^{(j-1)\gamma}}
\times\big\|f\big\|_{\mathrm{BMO}(\omega)}\underset{\xi\in8\mathcal{B}}{\mathrm{ess\,inf}}\,\omega(\xi)\\
&\lesssim\big\|f\big\|_{\mathrm{BMO}(\omega)}\underset{x\in8\mathcal{B}}{\mathrm{ess\,inf}}\,\omega(x).
\end{split}
\end{equation*}
Therefore,
\begin{equation*}
\begin{split}
J_{\infty}&=\frac{1}{\omega(\mathcal{B})}\int_{\mathcal{B}}\underset{y\in\mathcal{B}}{\mathrm{ess\,sup}}
\Big|\big[\mathcal{S}_{\infty}({f})(x)\big]-\big[\mathcal{S}_{\infty}({f})(y)\big]\Big|\,dx\\
&\lesssim\big\|f\big\|_{\mathrm{BMO}(\omega)}
\bigg\{\frac{m(\mathcal{B})}{\omega(\mathcal{B})}\cdot\underset{x\in8\mathcal{B}}{\mathrm{ess\,inf}}\,\omega(x)\bigg\}\\
&\leq\big\|f\big\|_{\mathrm{BMO}(\omega)}
\bigg\{\frac{m(\mathcal{B})}{\omega(\mathcal{B})}\cdot\underset{x\in\mathcal{B}}{\mathrm{ess\,inf}}\,\omega(x)\bigg\}
\leq\big\|f\big\|_{\mathrm{BMO}(\omega)},
\end{split}
\end{equation*}
where in the last step we have used inequality \eqref{omegawh}. Combining the above estimates for both terms $J_0$ and $J_{\infty}$ yields the desired result \eqref{mainesti2}. This concludes the proof of Theorem \ref{mainthm2}.
\end{proof}

\section{Proof of Theorem \ref{mainthm3}}
In this section, we will give the proof of Theorem \ref{mainthm3}. We first establish the following weighted $L^2$-boundedness of $\mathcal{G}^{\ast}_{\lambda}$, when $\lambda>2$ and $\omega\in A_2$.

\begin{thm}\label{thmlafunction}
Let $\omega\in A_2$. If $\lambda>2$, then the operator $\mathcal{G}^{\ast}_{\lambda}$ is bounded on $L^2(\omega)$.
\end{thm}

\begin{proof}
From the definition of $\mathcal{G}^{\ast}_{\lambda}$, we readily see that
\begin{equation*}
\begin{split}
\big[\mathcal{G}^{\ast}_{\lambda}({f})(x)\big]^2&=\int_0^{\infty}\int_{\mathbb R^n}
\Big(\frac{t}{t+|x-z|}\Big)^{\lambda n}\big|\psi_t*f(z)\big|^2\frac{dzdt}{t^{n+1}}\\
&=\int_0^{\infty}\int_{|x-z|<t}\Big(\frac{t}{t+|x-z|}\Big)^{\lambda n}
\big|\psi_t*f(z)\big|^2\frac{dzdt}{t^{n+1}}\\
&+\sum_{\ell=1}^{\infty}\int_0^{\infty}\int_{2^{\ell-1}t\leq|x-z|<2^{\ell}t}\Big(\frac{t}{t+|x-z|}\Big)^{\lambda n}
\big|\psi_t*f(z)\big|^2\frac{dzdt}{t^{n+1}}\\
&\lesssim \big[\mathcal{S}({f})(x)\big]^2+\sum_{\ell=1}^{\infty}2^{-\ell\lambda n}\big[\mathcal{S}_{2^{\ell}}({f})(x)\big]^2,
\end{split}
\end{equation*}
where
\begin{equation*}
\big[\mathcal{S}_{2^{\ell}}({f})(x)\big]
:=\bigg(\int_0^{\infty}\int_{|x-z|<2^{\ell}t}\big|\psi_t*f(z)\big|^2\frac{dzdt}{t^{n+1}}\bigg)^{1/2}.
\end{equation*}
Since $\omega\in A_2$, by part (2) of Lemma \ref{lemmaA1}, we get
\begin{equation*}
\omega\big(B(z,2^{\ell}t)\big)
=\omega\big(2^{\ell}{B}(z,t)\big)\leq 2^{2\ell n}[\omega]_{A_2}\cdot\omega\big({B}(z,t)\big),
\end{equation*}
which means that
\begin{equation*}
\int_{|z-x|<2^{\ell}t}\omega(x)\,dx\leq 2^{2\ell n}[\omega]_{A_2}\cdot\int_{|z-x|<t}\omega(x)\,dx.
\end{equation*}
This estimate, together with Fubini's theorem, implies that for all $f\in L^2(\omega)$,
\begin{equation*}
\begin{split}
\big\|\mathcal{S}_{2^{\ell}}(f)\big\|^2_{L^2(\omega)}
&=\int_{\mathbb R^n}\bigg(\int_0^{\infty}\int_{|z-x|<2^{\ell}t}
\big|\psi_t*f(z)\big|^2\frac{dzdt}{t^{n+1}}\bigg)\omega(x)\,dx\\
&=\iint_{\mathbb{R}^{n+1}_{+}}\bigg(\int_{\mathbb R^n}\chi_{|z-x|<2^{\ell}t}\omega(x)\,dx\bigg)
\big|\psi_t*f(z)\big|^2\frac{dzdt}{t^{n+1}}\\
&\leq2^{2\ell n}[\omega]_{A_2}\iint_{\mathbb{R}^{n+1}_{+}}\bigg(\int_{|z-x|<t}\omega(x)\,dx\bigg)
\big|\psi_t*f(z)\big|^2\frac{dzdt}{t^{n+1}}\\
&=2^{2\ell n}[\omega]_{A_2}\big\|\mathcal{S}(f)\big\|_{L^2(\omega)}^2.
\end{split}
\end{equation*}
Here $\chi_{|z-x|<2^{\ell}t}$ denotes the characteristic function of the set $\{x:|z-x|<2^{\ell}t\}$. Consequently, by Theorem \ref{thmsfunction} with $p=2$ and the fact that $\omega\in A_2$, we get
\begin{equation*}
\big\|\mathcal{G}^{\ast}_{\lambda}({f})\big\|^2_{L^2(\omega)}
\leq C\bigg[\big\|f\big\|_{L^2(\omega)}^2+\sum_{\ell=1}^{\infty}2^{-\ell\lambda n}2^{2\ell n}\big\|f\big\|_{L^2(\omega)}^2\bigg],
\end{equation*}
and hence, when $\lambda>2$, there exists a positive constant $C>0$, independent of $f$, such that
\begin{equation*}
\big\|\mathcal{G}^{\ast}_{\lambda}({f})\big\|_{L^2(\omega)}\leq C\big\|f\big\|_{L^2(\omega)}.
\end{equation*}
We are done.
\end{proof}

\begin{proof}[Proof of Theorem $\ref{mainthm3}$]
Let ${f}\in \mathrm{BMO}(\omega)$, $\omega\in A_1$ and $\lambda>3+{(2\delta+2\gamma)}/n$. Let $B(x_0,r)\subset\mathbb R^n$ be the open ball of radius $r\in(0,\infty)$ centered at $x_0\in \mathbb R^n$. Similar to the proofs of Theorems \ref{mainthm1} and \ref{mainthm2}, by the definition of $\mathrm{BLO}(\omega)$, we are going to prove that for any given ball $\mathcal{B}=B(x_0,r)$, the following inequality holds:
\begin{equation}\label{mainesti3}
\frac{1}{\omega(\mathcal{B})}\int_{\mathcal{B}}
\Big[\big[\mathcal{G}^{\ast}_{\lambda}({f})(x)\big]-
\underset{y\in\mathcal{B}}{\mathrm{ess\,inf}}\,\big[\mathcal{G}^{\ast}_{\lambda}({f})(y)\big]\Big]\,dx
\lesssim\big\|f\big\|_{\mathrm{BMO}(\omega)}.
\end{equation}
In order to prove \eqref{mainesti3}, we decompose the integral defining $\mathcal{G}^{\ast}_{\lambda}({f})$ into two parts.
\begin{equation*}
\begin{split}
\big[\mathcal{G}^{\ast}_{\lambda}({f})(x)\big]&=\bigg(\int_0^{\infty}\int_{\mathbb R^n}
\Big(\frac{t}{t+|x-z|}\Big)^{\lambda n}
\big|\psi_t*f(z)\big|^2\frac{dzdt}{t^{n+1}}\bigg)^{1/2}\\
\end{split}
\end{equation*}
\begin{equation*}
\begin{split}
&\leq\bigg(\int_0^{r}\int_{\mathbb R^n}
\Big(\frac{t}{t+|x-z|}\Big)^{\lambda n}\big|\psi_t*f(z)\big|^2\frac{dzdt}{t^{n+1}}\bigg)^{1/2}\\
&+\bigg(\int_{r}^{\infty}\int_{\mathbb R^n}
\Big(\frac{t}{t+|x-z|}\Big)^{\lambda n}\big|\psi_t*f(z)\big|^2\frac{dzdt}{t^{n+1}}\bigg)^{1/2}\\
&:=\big[\mathcal{G}^{\ast}_{\lambda,0}({f})(x)\big]+\big[\mathcal{G}^{\ast}_{\lambda,\infty}({f})(x)\big].
\end{split}
\end{equation*}
It is easy to see that
\begin{equation*}
\big[\mathcal{G}^{\ast}_{\lambda,0}({f})(x)\big],\big[\mathcal{G}^{\ast}_{\lambda,\infty}({f})(x)\big]
\leq\big[\mathcal{G}^{\ast}_{\lambda}({f})(x)\big]
\leq\big[\mathcal{G}^{\ast}_{\lambda,0}({f})(x)\big]+\big[\mathcal{G}^{\ast}_{\lambda,\infty}({f})(x)\big].
\end{equation*}
Consequently, in view of \eqref{essinf}, we can deduce that
\begin{equation*}
\begin{split}
&\frac{1}{\omega(\mathcal{B})}\int_{\mathcal{B}}
\Big[\big[\mathcal{G}^{\ast}_{\lambda}({f})(x)\big]-
\underset{y\in\mathcal{B}}{\mathrm{ess\,inf}}\,\big[\mathcal{G}^{\ast}_{\lambda}({f})(y)\big]\Big]\,dx\\
&\leq\frac{1}{\omega(\mathcal{B})}\int_{\mathcal{B}}
\Big[\big[\mathcal{G}^{\ast}_{\lambda,0}({f})(x)\big]+\big[\mathcal{G}^{\ast}_{\lambda,\infty}({f})(x)\big]
-\underset{y\in\mathcal{B}}{\mathrm{ess\,inf}}\,\big[\mathcal{G}^{\ast}_{\lambda}({f})(y)\big]\Big]\,dx\\
&\leq\frac{1}{\omega(\mathcal{B})}\int_{\mathcal{B}}
\Big[\big[\mathcal{G}^{\ast}_{\lambda,0}({f})(x)\big]+\big[\mathcal{G}^{\ast}_{\lambda,\infty}({f})(x)\big]
-\underset{y\in\mathcal{B}}{\mathrm{ess\,inf}}\,\big[\mathcal{G}^{\ast}_{\lambda,\infty}({f})(y)\big]\Big]\,dx\\
&\leq\frac{1}{\omega(\mathcal{B})}\int_{\mathcal{B}}\big[\mathcal{G}^{\ast}_{\lambda,0}({f})(x)\big]dx
+\frac{1}{\omega(\mathcal{B})}\int_{\mathcal{B}}
\underset{y\in\mathcal{B}}{\mathrm{ess\,sup}}\Big|
\big[\mathcal{G}^{\ast}_{\lambda,\infty}({f})(x)\big]-\big[\mathcal{G}^{\ast}_{\lambda,\infty}({f})(y)\big]\Big|\,dx\\
&:=\mathcal{K}_0+\mathcal{K}_{\infty}.
\end{split}
\end{equation*}
As before, to control the first term $\mathcal{K}_0$, we use the decomposition
\begin{equation*}
\begin{split}
f(x)&=f_{4\mathcal{B}}+[f(x)-f_{4\mathcal{B}}]\cdot\chi_{4\mathcal{B}}(x)
+[f(x)-f_{4\mathcal{B}}]\cdot\chi_{(4\mathcal{B})^{\complement}}(x)\\
&:=f^1(x)+f^2(x)+f^3(x).
\end{split}
\end{equation*}
By the vanishing condition of the kernel $\mathcal{\psi}$, we have that for any $x\in\mathbb R^n$ and $t>0$,
\begin{equation*}
\psi_t*f(x)=\psi_t*f^2(x)+\psi_t*f^3(x)\quad \& \quad \big[\mathcal{G}^{\ast}_{\lambda,0}({f}^1)(x)\big]\equiv0.
\end{equation*}
Then we can write
\begin{equation*}
\begin{split}
\mathcal{K}_0&=\frac{1}{\omega(\mathcal{B})}\int_{\mathcal{B}}\big[\mathcal{G}^{\ast}_{\lambda,0}({f})(x)\big]dx\\
&\leq\frac{1}{\omega(\mathcal{B})}\int_{\mathcal{B}}\big[\mathcal{G}^{\ast}_{\lambda,0}({f}^2)(x)\big]dx
+\frac{1}{\omega(\mathcal{B})}\int_{\mathcal{B}}\big[\mathcal{G}^{\ast}_{\lambda,0}({f}^3)(x)\big]dx\\
&:=\mathcal{K}^2_0+\mathcal{K}^3_0.
\end{split}
\end{equation*}
Applying Theorem \ref{thmlafunction},Cauchy--Schwarz's inequality and Lemma \ref{BMOp}, we have
\begin{equation*}
\begin{split}
\mathcal{K}^2_0
&\leq\frac{1}{\omega(\mathcal{B})}\bigg(\int_{\mathcal{B}}\big[\mathcal{G}^{\ast}_{\lambda,0}({f}^2)(x)\big]^2\omega(x)^{-1}dx\bigg)^{1/2}
\bigg(\int_{\mathcal{B}}\omega(x)\,dx\bigg)^{1/2}\\
&\leq\frac{1}{\omega(\mathcal{B})^{1/2}}\big\|\mathcal{G}^{\ast}_{\lambda}({f}^2)\big\|_{L^2(\omega^{-1})}\\
&\leq C\cdot\frac{\omega(4\mathcal{B})^{1/2}}{\omega(\mathcal{B})^{1/2}}
\bigg(\frac{1}{\omega(4\mathcal{B})}\int_{4\mathcal{B}}\big|f(x)-f_{4\mathcal{B}}\big|^2\omega(x)^{-1}dx\bigg)^{1/2}\\
&\leq C\big\|f\big\|_{\mathrm{BMO}(\omega)}.
\end{split}
\end{equation*}
Let us now estimate the other term $\mathcal{K}^3_0$. First, by the size condition of the kernel $\psi$, we can see that for any $z\in\mathbb R^n$ and $0<t<r$,
\begin{equation*}
\begin{split}
&\big|\psi_t*f^3(z)\big|
=\bigg|\int_{\mathbb R^n\setminus 4\mathcal{B}}\psi_t(z-\xi)\big[f(\xi)-f_{4\mathcal{B}}\big]\,d\xi\bigg|\\
&\lesssim\int_{\mathbb R^n\setminus 4\mathcal{B}}\frac{t^{\delta}}{(t+|z-\xi|)^{n+\delta}}
\big|f(\xi)-f_{4\mathcal{B}}\big|\,d\xi.\\
\end{split}
\end{equation*}
Next, by Minkowski's inequality for integrals,
\begin{equation*}
\begin{split}
\big[\mathcal{G}^{\ast}_{\lambda,0}({f}^3)(x)\big]
&\lesssim\int_{\mathbb R^n\setminus 4\mathcal{B}}\big|f(\xi)-f_{4\mathcal{B}}\big|\,d\xi\\
&\times\bigg(\int_0^{r}\int_{\mathbb R^n}\Big(\frac{t}{t+|x-z|}\Big)^{\lambda n}
\frac{t^{2\delta}}{(t+|z-\xi|)^{2(n+\delta)}}\frac{dzdt}{t^{n+1}}\bigg)^{1/2}.
\end{split}
\end{equation*}
We now claim that
\begin{equation}\label{hard1}
\mathrm{I}:=\int_{\mathbb R^n}\left(\frac{t}{t+|x-z|}\right)^{\lambda n}\frac{t^{2\delta-n}}{(t+|z-\xi|)^{2(n+\delta)}}dz
\lesssim\frac{t^{2\delta}}{(|x_0-\xi|)^{2(n+\delta)}}.
\end{equation}
In fact, we have
\begin{equation*}
\begin{split}
\mathrm{I}&=\int_{|x-z|<t}\left(\frac{t}{t+|x-z|}\right)^{\lambda n}\frac{t^{2\delta-n}}{(t+|z-\xi|)^{2(n+\delta)}}dz\\
&+\sum_{\ell=1}^{\infty}\int_{2^{\ell-1}t\leq|x-z|<2^{\ell}t}\left(\frac{t}{t+|x-z|}\right)^{\lambda n}\frac{t^{2\delta-n}}{(t+|z-\xi|)^{2(n+\delta)}}dz\\
&\leq\int_{|x-z|<t}\frac{t^{2\delta-n}}{(t+|z-\xi|)^{2(n+\delta)}}dz\\
&+\sum_{\ell=1}^{\infty}\int_{2^{\ell-1}t\leq|x-z|<2^{\ell}t}\Big(\frac{1}{2^{\ell}}\Big)^{\lambda n}\frac{t^{2\delta-n}}{(t+|z-\xi|)^{2(n+\delta)}}dz.
\end{split}
\end{equation*}
We do our estimates according to different cases.
\begin{itemize}
  \item When $|x-z|<t$ with $0<t<r$, $x\in\mathcal{B}=B(x_0,r)$ and $\xi\in\mathbb R^n\setminus (4\mathcal{B})$, one has
  \begin{equation*}
  |z-x_0|\leq |z-x|+|x-x_0|<2r\leq\frac{|x_0-\xi|}{2},
  \end{equation*}
  and
  \begin{equation*}
  |z-\xi|\geq |x_0-\xi|-|z-x_0|>\frac{|x_0-\xi|}{2}.
  \end{equation*}
  Hence,
  \begin{equation*}
  \begin{split}
  \int_{|x-z|<t}\frac{t^{2\delta-n}}{(t+|z-\xi|)^{2(n+\delta)}}dz
  &\leq \int_{|x-z|<t}\frac{t^{2\delta-n}}{(|z-\xi|)^{2(n+\delta)}}dz\\
  &\lesssim\frac{t^{2\delta}}{(|x_0-\xi|)^{2(n+\delta)}},
  \end{split}
  \end{equation*}
  as desired.
  \item When $|x-z|<2^{\ell}t$ with $0<t<r$ and $\ell\geq1$, $x\in\mathcal{B}=B(x_0,r)$ and $\xi\in\mathbb R^n\setminus (2^{\ell+2}\mathcal{B})$, in this case, one has
  \begin{equation*}
  |z-x_0|\leq |z-x|+|x-x_0|<2^{\ell}r+r\leq 2^{\ell+1}r\leq\frac{|x_0-\xi|}{2},
  \end{equation*}
  and
  \begin{equation*}
  |z-\xi|\geq |x_0-\xi|-|z-x_0|>\frac{|x_0-\xi|}{2}.
  \end{equation*}
  Hence, for any fixed $\ell\in \mathbb{N}$,
  \begin{equation*}
  \begin{split}
  &\int_{2^{\ell-1}t\leq|x-z|<2^{\ell}t}\Big(\frac{1}{2^{\ell}}\Big)^{\lambda n}\frac{t^{2\delta-n}}{(t+|z-\xi|)^{2(n+\delta)}}dz\\
  &\leq\int_{2^{\ell-1}t\leq|x-z|<2^{\ell}t}\Big(\frac{1}{2^{\ell}}\Big)^{\lambda n}
  \frac{t^{2\delta-n}}{(|z-\xi|)^{2(n+\delta)}}dz\\
  &\lesssim\Big(\frac{1}{2^{\ell}}\Big)^{(\lambda-1)n}\frac{t^{2\delta}}{(|x_0-\xi|)^{2(n+\delta)}}.
  \end{split}
  \end{equation*}
  \item When $|x-z|<2^{\ell}t$ with $0<t<r$ and $\ell\geq1$, $x\in\mathcal{B}=B(x_0,r)$ and $\xi\in\mathbb (2^{\ell+2}\mathcal{B})\setminus (4\mathcal{B})$, in this case, one has
  \begin{equation*}
  |\xi-x_0|\geq4r\Longrightarrow r\leq\frac{|\xi-x_0|}{4},
  \end{equation*}
  and
  \begin{equation*}
  \begin{split}
  |x_0-\xi|&\leq |x_0-x|+|x-z|+|z-\xi|\\
  &<r+2^{\ell}t+|z-\xi|<\frac{|x_0-\xi|}{4}+2^{\ell}\big(t+|z-\xi|\big),
  \end{split}
  \end{equation*}
  which in turn implies that
  \begin{equation*}
  |x_0-\xi|<\frac{4}{\,3\,}\cdot2^{\ell}\big(t+|z-\xi|\big).
  \end{equation*}
  Thus, for any fixed $\ell\in \mathbb{N}$, it follows that
  \begin{equation*}
  \begin{split}
  &\int_{2^{\ell-1}t\leq|x-z|<2^{\ell}t}\Big(\frac{1}{2^{\ell}}\Big)^{\lambda n}\frac{t^{2\delta-n}}{(t+|z-\xi|)^{2(n+\delta)}}dz\\
  &\lesssim\int_{2^{\ell-1}t\leq|x-z|<2^{\ell}t}\Big(\frac{1}{2^{\ell}}\Big)^{\lambda n}
  \big(2^{\ell}\big)^{2(n+\delta)}\frac{t^{2\delta-n}}{(|x_0-\xi|)^{2(n+\delta)}}dz\\
  &\lesssim\Big(\frac{1}{2^{\ell}}\Big)^{\lambda n-3n-2\delta}\frac{t^{2\delta}}{(|x_0-\xi|)^{2(n+\delta)}}.
  \end{split}
  \end{equation*}
\end{itemize}
Summing up the estimates derived above, we conclude that
\begin{equation*}
\begin{split}
\mathrm{I}&\lesssim\frac{t^{2\delta}}{(|x_0-\xi|)^{2(n+\delta)}}+\sum_{\ell=1}^{\infty}
\Big[\Big(\frac{1}{2^{\ell}}\Big)^{(\lambda-1)n}+
\Big(\frac{1}{2^{\ell}}\Big)^{\lambda n-3n-2\delta}\Big]
\frac{t^{2\delta}}{(|x_0-\xi|)^{2(n+\delta)}}\\
&\lesssim\frac{t^{2\delta}}{(|x_0-\xi|)^{2(n+\delta)}},
\end{split}
\end{equation*}
where in the last step we have used the fact that $\lambda>3+{(2\delta)}/n>1$. This proves \eqref{hard1}. From this estimate and part (2) of Lemma \ref{wanglemma1}, it follows that
\begin{equation*}
\begin{split}
\mathcal{K}^3_0&\lesssim \frac{m(\mathcal{B})}{\omega(\mathcal{B})}\cdot
\int_{\mathbb R^n\setminus 4\mathcal{B}}\frac{1}{|x_0-\xi|^{n+\delta}}\big|f(\xi)-f_{4\mathcal{B}}\big|\,d\xi
\bigg(\int_0^{r}t^{2\delta-1}dt\bigg)^{1/2}\\
&\lesssim\frac{m(\mathcal{B})}{\omega(\mathcal{B})}\cdot
\sum_{j=2}^{\infty}\int_{2^{j+1}\mathcal{B}\setminus2^j \mathcal{B}}
\frac{r^{\delta}}{|x_0-\xi|^{n+\delta}}\big|f(\xi)-f_{4\mathcal{B}}\big|\,d\xi\\
&\lesssim \frac{m(\mathcal{B})}{\omega(\mathcal{B})}\cdot\sum_{j=2}^{\infty}\int_{2^{j+1}\mathcal{B}}
\frac{r^{\delta}}{m(2^{j+1}\mathcal{B})^{1+\delta/n}}\big|f(\xi)-f_{4\mathcal{B}}\big|\,d\xi\\
&\lesssim r^{\delta}\bigg\{\sum_{j=2}^{\infty}\frac{j}{(2^j)^{\delta}r^{\delta}}\bigg\}
\times\big\|f\big\|_{\mathrm{BMO}(\omega)}
\bigg\{\frac{m(\mathcal{B})}{\omega(\mathcal{B})}\cdot\underset{x\in\mathcal{B}}{\mathrm{ess\,inf}}\,\omega(x)\bigg\}\\
&\lesssim\big\|f\big\|_{\mathrm{BMO}(\omega)},
\end{split}
\end{equation*}
where the last inequality follows from \eqref{omegawh}. On the other hand, by \eqref{regularity2}, \eqref{firstsecond4} and the vanishing condition of the kernel $\psi$, we can see that for any $x,y\in \mathcal{B}=B(x_0,r)$ and $t>r$,
\begin{equation}\label{firstsecond6}
\begin{split}
&\Big|\big|\psi_t*f(x+z)\big|-\big|\psi_t*f(y+z)\big|\Big|\leq\Big|\psi_t*f(x+z)-\psi_t*f(y+z)\Big|\\
&=\bigg|\int_{\mathbb R^n}\big[\psi_t(x+z-\xi)-\psi_t(y+z-\xi)\big]\cdot\big[f(\xi)-f_{8\mathcal{B}}\big]\,d\xi\bigg|\\
&\lesssim\int_{8\mathcal{B}}\big[\big|\psi_t(x+z-\xi)\big|+\big|\psi_t(y+z-\xi)\big|\big]\big|f(\xi)-f_{8\mathcal{B}}\big|\,d\xi\\
&+\int_{\mathbb R^n\setminus8\mathcal{B}}\frac{t^{\delta}\cdot|x-y|^{\gamma}}{(t+|x+z-\xi|)^{n+\delta+\gamma}}
\big|f(\xi)-f_{8\mathcal{B}}\big|\,d\xi.
\end{split}
\end{equation}
Consequently, by \eqref{firstsecond6} and the Minkowski inequality for integrals, we obtain that for any $x,y\in\mathcal{B}=B(x_0,r)$,
\begin{equation}\label{lastterm}
\begin{split}
&\Big|\big[\mathcal{G}^{\ast}_{\lambda,\infty}({f})(x)\big]-\big[\mathcal{G}^{\ast}_{\lambda,\infty}({f})(y)\big]\Big|\\
&=\bigg|\bigg(\int_{r}^{\infty}\int_{\mathbb R^n}\Big(\frac{t}{t+|z|}\Big)^{\lambda n}
\big|\psi_t*f(x+z)\big|^2\frac{dzdt}{t^{n+1}}\bigg)^{1/2}\\
&-\bigg(\int_{r}^{\infty}\int_{\mathbb R^n}\Big(\frac{t}{t+|z|}\Big)^{\lambda n}
\big|\psi_t*f(y+z)\big|^2\frac{dzdt}{t^{n+1}}\bigg)^{1/2}\bigg|\\
&\leq\bigg(\int_{r}^{\infty}\int_{\mathbb R^n}\Big(\frac{t}{t+|z|}\Big)^{\lambda n}
\Big|\big|\psi_t*f(x+z)\big|-\big|\psi_t*f(y+z)\big|\Big|^2\frac{dzdt}{t^{n+1}}\bigg)^{1/2}\\
&\lesssim\bigg(\int_{r}^{\infty}\int_{\mathbb R^n}\Big(\frac{t}{t+|z|}\Big)^{\lambda n}
\bigg\{\int_{8\mathcal{B}}\big[\big|\psi_t(x+z-\xi)\big|+\big|\psi_t(y+z-\xi)\big|\big]\big|f(\xi)-f_{8\mathcal{B}}\big|\,d\xi\bigg\}^2
\frac{dzdt}{t^{n+1}}\bigg)^{1/2}\\
&+\bigg(\int_{r}^{\infty}\int_{\mathbb R^n}\Big(\frac{t}{t+|z|}\Big)^{\lambda n}
\bigg\{t^{\delta}\cdot\int_{\mathbb R^n\setminus8\mathcal{B}}\frac{(2r)^{\gamma}}{(t+|x+z-\xi|)^{n+\delta+\gamma}}
\big|f(\xi)-f_{8\mathcal{B}}\big|\,d\xi\bigg\}^2
\frac{dzdt}{t^{n+1}}\bigg)^{1/2}.
\end{split}
\end{equation}
In view of \eqref{A83}, we have that for any $t>0$,
\begin{equation*}
\big|\psi_t(x+z-\xi)\big|+\big|\psi_t(y+z-\xi)\big|\lesssim\frac{1}{t^n}.
\end{equation*}
Moreover,
\begin{equation*}
\begin{split}
&\bigg\{\int_{r}^{\infty}\int_{\mathbb R^n}\left(\frac{t}{t+|z|}\right)^{\lambda n}\frac{1}{t^{3n+1}}\,dzdt\bigg\}^{1/2}\\
&=\bigg\{\int_{r}^{\infty}\int_{|z|<t}\left(\frac{t}{t+|z|}\right)^{\lambda n}\frac{1}{t^{3n+1}}\,dzdt
+\int_{r}^{\infty}\int_{|z|\geq t}\left(\frac{t}{t+|z|}\right)^{\lambda n}\frac{1}{t^{3n+1}}\,dzdt\bigg\}^{1/2}\\
&\leq\bigg\{\int_{r}^{\infty}\int_{|z|<t}\frac{1}{t^{3n+1}}\,dzdt
+\int_{r}^{\infty}\int_{|z|\geq t}\left(\frac{t}{|z|}\right)^{\lambda n}\frac{1}{t^{3n+1}}\,dzdt\bigg\}^{1/2}.
\end{split}
\end{equation*}
Using polar coordinates and the assumption $\lambda>1$, we get
\begin{equation*}
\begin{split}
\int_{r}^{\infty}\int_{|z|\geq t}\left(\frac{t}{|z|}\right)^{\lambda n}\frac{1}{t^{3n+1}}\,dzdt
&\lesssim\int_{r}^{\infty}\bigg(\int_{t}^{\infty}\frac{\varrho^{n-1}}{\varrho^{\lambda n}}\,d\varrho\bigg)\frac{t^{\lambda n}}{t^{3n+1}}\,dt\\
&\lesssim\int_{r}^{\infty}\frac{1}{t^{2n+1}}\,dt\lesssim\frac{1}{r^{2n}}.
\end{split}
\end{equation*}
In addition,
\begin{equation*}
\int_{r}^{\infty}\int_{|z|<t}\frac{1}{t^{3n+1}}\,dzdt
\lesssim\int_{r}^{\infty}\frac{1}{t^{2n+1}}\,dt\lesssim\frac{1}{r^{2n}}.
\end{equation*}
Thus it follows that
\begin{equation*}
\bigg\{\int_{r}^{\infty}\int_{\mathbb R^n}\left(\frac{t}{t+|z|}\right)^{\lambda n}\frac{1}{t^{3n+1}}\,dzdt\bigg\}^{1/2}
\lesssim\frac{1}{r^{n}}.
\end{equation*}
Hence, by the Minkowski inequality for integrals, we can see that the first term in \eqref{lastterm} is bounded by
\begin{equation*}
\begin{split}
&\int_{8\mathcal{B}}\big|f(\xi)-f_{8\mathcal{B}}\big|\,d\xi
\bigg\{\int_{r}^{\infty}\int_{\mathbb R^n}\left(\frac{t}{t+|z|}\right)^{\lambda n}\frac{1}{t^{3n+1}}\,dzdt\bigg\}^{1/2}\\
&\lesssim\frac{1}{m(8\mathcal{B})}\int_{8\mathcal{B}}\big|f(\xi)-f_{8\mathcal{B}}\big|\,d\xi.
\end{split}
\end{equation*}
We now claim that
\begin{equation}\label{hard2}
\mathrm{II}:=\int_{\mathbb R^n}\left(\frac{t}{t+|z|}\right)^{\lambda n}
\frac{t^{2\delta-n}}{(t+|x+z-\xi|)^{2(n+\delta+\gamma)}}dz
\lesssim\frac{t^{2\delta}}{(t+|x_0-\xi|)^{2(n+\delta+\gamma)}}.
\end{equation}
In fact, we have
\begin{equation*}
\begin{split}
\mathrm{II}&=\int_{|z|<t}\left(\frac{t}{t+|z|}\right)^{\lambda n}\frac{t^{2\delta-n}}{(t+|x+z-\xi|)^{2(n+\delta+\gamma)}}dz\\
&+\sum_{\ell=1}^{\infty}\int_{2^{\ell-1}t\leq|z|<2^{\ell}t}
\left(\frac{t}{t+|z|}\right)^{\lambda n}\frac{t^{2\delta-n}}{(t+|x+z-\xi|)^{2(n+\delta+\gamma)}}dz\\
&\leq\int_{|z|<t}\frac{t^{2\delta-n}}{(t+|x+z-\xi|)^{2(n+\delta+\gamma)}}dz\\
&+\sum_{\ell=1}^{\infty}\int_{2^{\ell-1}t\leq|z|<2^{\ell}t}
\Big(\frac{1}{2^{\ell}}\Big)^{\lambda n}\frac{t^{2\delta-n}}{(t+|x+z-\xi|)^{2(n+\delta+\gamma)}}dz.
\end{split}
\end{equation*}
\begin{itemize}
  \item When $|z|<t$ with $t\geq r$, $x\in\mathcal{B}=B(x_0,r)$ and $\xi\in \mathbb R^n\setminus(8\mathcal{B})$, one has
  \begin{equation*}
  |x_0-x-z|\leq |x_0-x|+|z|<r+t\leq 2t,
  \end{equation*}
  and
  \begin{equation*}
  t+|x_0-\xi|\leq t+|x_0-x-z|+|x+z-\xi|\leq3(t+|x+z-\xi|).
  \end{equation*}
  Hence,
  \begin{equation*}
  \begin{split}
  \int_{|z|<t}\frac{t^{2\delta-n}}{(t+|x+z-\xi|)^{2(n+\delta+\gamma)}}dz
  &\lesssim\int_{|z|<t}\frac{t^{2\delta-n}}{(t+|x_0-\xi|)^{2(n+\delta+\gamma)}}dz\\
  &\lesssim\frac{t^{2\delta}}{(t+|x_0-\xi|)^{2(n+\delta+\gamma)}},
  \end{split}
  \end{equation*}
  as desired.
  \item When $|z|<2^{\ell}t$ with $t\geq r$ and $\ell\geq1$, $x\in\mathcal{B}=B(x_0,r)$ and $\xi\in\mathbb R^n\setminus (8\mathcal{B})$, in this case, one has
  \begin{equation*}
  |x_0-x-z|\leq |x_0-x|+|z|<r+2^{\ell}t\leq t+2^{\ell}t,
  \end{equation*}
  and
  \begin{equation*}
  \begin{split}
  t+|x_0-\xi|&\leq t+|x_0-x-z|+|x+z-\xi|\\
  &<2t+2^{\ell}t+|x+z-\xi|\leq 2^{\ell+1}\big(t+|x+z-\xi|\big).
  \end{split}
  \end{equation*}
  Thus, for any fixed $\ell\in \mathbb{N}$, it follows that
  \begin{equation*}
  \begin{split}
&\int_{2^{\ell-1}t\leq|z|<2^{\ell}t}
\Big(\frac{1}{2^{\ell}}\Big)^{\lambda n}\frac{t^{2\delta-n}}{(t+|x+z-\xi|)^{2(n+\delta+\gamma)}}dz\\
&\lesssim\int_{2^{\ell-1}t\leq|x-z|<2^{\ell}t}\Big(\frac{1}{2^{\ell}}\Big)^{\lambda n}
  \big(2^{\ell}\big)^{2(n+\delta+\gamma)}\frac{t^{2\delta-n}}{(t+|x_0-\xi|)^{2(n+\delta+\gamma)}}dz\\
  &\lesssim\Big(\frac{1}{2^{\ell}}\Big)^{\lambda n-3n-2\delta-2\gamma}\frac{t^{2\delta}}{(t+|x_0-\xi|)^{2(n+\delta+\gamma)}}.
  \end{split}
  \end{equation*}
\end{itemize}
Summing up the estimates derived above, we conclude that
\begin{equation*}
\begin{split}
\mathrm{II}&\lesssim\frac{t^{2\delta}}{(t+|x_0-\xi|)^{2(n+\delta+\gamma)}}+\sum_{\ell=1}^{\infty}
\Big[\Big(\frac{1}{2^{\ell}}\Big)^{\lambda n-3n-2\delta-2\gamma}\Big]
\frac{t^{2\delta}}{(t+|x_0-\xi|)^{2(n+\delta+\gamma)}}\\
&\lesssim\frac{t^{2\delta}}{(t+|x_0-\xi|)^{2(n+\delta+\gamma)}},
\end{split}
\end{equation*}
where in the last inequality we have used the assumption $\lambda>3+{(2\delta+2\gamma)}/n$. This proves \eqref{hard2}. Furthermore, by the Minkowski inequality for integrals, we can see that the second term in \eqref{lastterm} is dominated by
\begin{equation*}
\begin{split}
&\int_{\mathbb R^n\setminus8\mathcal{B}}(2r)^{\gamma}\big|f(\xi)-f_{8\mathcal{B}}\big|\,d\xi
\bigg\{\int_{r}^{\infty}\int_{\mathbb R^n}\left(\frac{t}{t+|z|}\right)^{\lambda n}
\frac{t^{2\delta}}{(t+|x+z-\xi|)^{2(n+\delta+\gamma)}}\frac{dzdt}{t^{n+1}}\bigg\}^{1/2}\\
&\lesssim\int_{\mathbb R^n\setminus8\mathcal{B}}(2r)^{\gamma}\big|f(\xi)-f_{8\mathcal{B}}\big|\,d\xi
\bigg\{\int_{r}^{\infty}\frac{t^{2\delta-1}}{(t+|x_0-\xi|)^{2(n+\delta+\gamma)}}\,dt\bigg\}^{1/2}\\
&\lesssim\int_{\mathbb R^n\setminus8\mathcal{B}}\frac{(2r)^{\gamma}}{(|x_0-\xi|)^{n+\gamma}}\big|f(\xi)-f_{8\mathcal{B}}\big|\,d\xi.
\end{split}
\end{equation*}
From part (1) and part (2) of Lemma \ref{wanglemma1}, it then follows that for any $x,y\in \mathcal{B}$,
\begin{equation*}
\begin{split}
&\Big|\big[\mathcal{G}^{\ast}_{\lambda,\infty}({f})(x)\big]-\big[\mathcal{G}^{\ast}_{\lambda,\infty}({f})(y)\big]\Big|
\lesssim\frac{1}{m(8\mathcal{B})}\int_{8\mathcal{B}}\big|f(\xi)-f_{8\mathcal{B}}\big|\,d\xi\\
&+\sum_{j=3}^{\infty}\int_{2^{j+1}\mathcal{B}\setminus2^j\mathcal{B}}
\frac{(2r)^{\gamma}}{(|x_0-\xi|)^{n+\gamma}}\big|f(\xi)-f_{8\mathcal{B}}\big|\,d\xi\\
&\lesssim\frac{1}{m(8\mathcal{B})}\int_{8\mathcal{B}}\big|f(\xi)-f_{8\mathcal{B}}\big|\,d\xi
+\sum_{j=3}^{\infty}\bigg\{\frac{(2r)^{\gamma}}{(2^jr)^{\gamma}}\cdot\frac{1}{m(2^{j+1}\mathcal{B})}
\int_{2^{j+1}\mathcal{B}}\big|f(\xi)-f_{8\mathcal{B}}\big|\,d\xi\bigg\}\\
&\lesssim\big\|f\big\|_{\mathrm{BMO}(\omega)}\underset{\xi\in 8\mathcal{B}}{\mathrm{ess\,inf}}\,\omega(\xi)
+\sum_{j=3}^{\infty}\frac{j}{2^{(j-1)\gamma}}
\times\big\|f\big\|_{\mathrm{BMO}(\omega)}\underset{\xi\in8\mathcal{B}}{\mathrm{ess\,inf}}\,\omega(\xi)\\
&\lesssim\big\|f\big\|_{\mathrm{BMO}(\omega)}\underset{x\in8\mathcal{B}}{\mathrm{ess\,inf}}\,\omega(x).
\end{split}
\end{equation*}
Therefore, by inequality \eqref{omegawh} again, we conclude that
\begin{equation*}
\begin{split}
\mathcal{K}_{\infty}&=\frac{1}{\omega(\mathcal{B})}\int_{\mathcal{B}}\underset{y\in\mathcal{B}}{\mathrm{ess\,sup}}
\Big|\big[\mathcal{G}^{\ast}_{\lambda,\infty}({f})(x)\big]-\big[\mathcal{G}^{\ast}_{\lambda,\infty}({f})(y)\big]\Big|\,dx\\
&\lesssim\big\|f\big\|_{\mathrm{BMO}(\omega)}
\bigg\{\frac{m(\mathcal{B})}{\omega(\mathcal{B})}\cdot\underset{x\in8\mathcal{B}}{\mathrm{ess\,inf}}\,\omega(x)\bigg\}\\
&\leq\big\|f\big\|_{\mathrm{BMO}(\omega)}
\bigg\{\frac{m(\mathcal{B})}{\omega(\mathcal{B})}\cdot\underset{x\in\mathcal{B}}{\mathrm{ess\,inf}}\,\omega(x)\bigg\}
\leq\big\|f\big\|_{\mathrm{BMO}(\omega)}.
\end{split}
\end{equation*}
Combining the above estimates for both terms $\mathcal{K}_0$ and $\mathcal{K}_{\infty}$ leads to the desired inequality \eqref{mainesti3}. This concludes the proof of Theorem \ref{mainthm3}.
\end{proof}

We now introduce the function space $L^{\infty}(\omega)$ (the weighted version of $L^{\infty}(\mathbb R^n)$), which denotes the space of all essentially bounded measurable functions with respect to $\omega$ on $\mathbb R^n$. More precisely, we say that a measurable function $f$ belongs to the weighted space $L^{\infty}(\omega)$, if there is a constant $C>0$ such that for any ball $\mathcal{B}$ in $\mathbb R^n$,
\begin{equation*}
\bigg(\underset{x\in \mathcal{B}}{\mathrm{ess\,inf}}\,\omega(x)\bigg)^{-1}\Big(\big\|f\big\|_{L^{\infty}(\mathcal{B})}\Big)\leq C<\infty,
\end{equation*}
where
\begin{equation*}
\big\|f\big\|_{L^{\infty}(\mathcal{B})}:=\underset{x\in \mathcal{B}}{\mbox{ess\,sup}}\,|f(x)|
=\inf\Big\{M>0:m\big(\big\{x\in \mathcal{B}:|f(x)|>M\big\}\big)=0\Big\}.
\end{equation*}
We define
\begin{equation*}
\big\|f\big\|_{L^{\infty}(\omega)}:=\sup_{\mathcal{B}\subset\mathbb R^n}
\bigg(\underset{x\in \mathcal{B}}{\mathrm{ess\,inf}}\,\omega(x)\bigg)^{-1}\cdot
\Big(\big\|f\big\|_{L^{\infty}(\mathcal{B})}\Big).
\end{equation*}
In particular, when $\omega(x)=1$, it is easy to see that
\begin{equation*}
\big\|f\big\|_{L^{\infty}(\omega)}=\sup_{\mathcal{B}\subset\mathbb R^n}\big\|f\big\|_{L^{\infty}(\mathcal{B})}
=\underset{x\in \mathbb R^n}{\mbox{ess\,sup}}\,|f(x)|=\big\|f\big\|_{L^\infty}.
\end{equation*}
For any ball $\mathcal{B}$ in $\mathbb R^n$, we have
\begin{equation*}
\begin{split}
\frac{1}{\omega(\mathcal{B})}\int_{\mathcal{B}}\big|f(x)-f_{\mathcal{B}}\big|\,dx
&\leq \frac{2}{\omega(\mathcal{B})}\int_{\mathcal{B}}|f(x)|\,dx\\
&\leq 2\frac{m(\mathcal{B})}{\omega(\mathcal{B})}\big\|f\big\|_{L^{\infty}(\mathcal{B})}\\
&\leq 2\big\|f\big\|_{L^{\infty}(\omega)},
\end{split}
\end{equation*}
and hence
\begin{equation*}
L^{\infty}(\omega)\subset \mathrm{BMO}(\omega)\quad \& \quad \big\|f\big\|_{\mathrm{BMO}(\omega)}\leq2\big\|f\big\|_{L^{\infty}(\omega)}.
\end{equation*}

By using similar arguments, the following results can be also derived from Theorems \ref{thmgfunction}, \ref{thmsfunction} and \ref{thmlafunction}, which are the weighted version of Leckband's result in \cite{leck}.
\begin{thm}
For any ${f}\in L^{\infty}(\omega)$ and $\omega\in A_1$, then $\mathcal{G}({f})$(or $\mathcal{S}({f})$) is finite everywhere, and there exists a positive constant $C>0$, independent of $f$, such that
\begin{equation*}
\big\|\mathcal{G}({f})\big\|_{\mathrm{BLO}(\omega)}\leq C\big\|f\big\|_{L^{\infty}(\omega)}
\quad \mbox{and} \quad \big\|\mathcal{S}({f})\big\|_{\mathrm{BLO}(\omega)}\leq C\big\|f\big\|_{L^{\infty}(\omega)}.
\end{equation*}
\end{thm}

\begin{thm}
Suppose that $(\lambda-3)n>2\delta+2\gamma$. For any ${f}\in L^{\infty}(\omega)$ and $\omega\in A_1$, then $\mathcal{G}^{\ast}_{\lambda}({f})$ is finite everywhere, and there exists a positive constant $C>0$, independent of $f$, such that
\begin{equation*}
\big\|\mathcal{G}^{\ast}_{\lambda}({f})\big\|_{\mathrm{BLO}(\omega)}\leq C\big\|f\big\|_{L^{\infty}(\omega)}.
\end{equation*}
\end{thm}

\section{John--Nirenberg-type inequalities for weighted BLO and BMO spaces}

In this section, we are concerned with the John--Nirenberg-type inequality with precise constants suitable for the $\mathrm{BLO}(\omega)$ spaces and relevant properties. Before we proceed to prove our next theorems, we need the following two auxiliary lemmas, which can be found in \cite{duoand,grafakos} and \cite{li}. Here we give the proofs for the sake of completeness.

\begin{lem}\label{lem71}
Let $n\in\mathbb N$, $1\leq p<\infty$ and $\omega$ be a weight function on $\mathbb R^n$. Then for any measurable subset $E$ of $\mathbb R^n$(with finite or infinite Lebesgue measure), we have the following integration formula for $f\in L^p(\omega)$(the layer cake representation).
\begin{equation}\label{layer}
\int_{E}|f(x)|^p\omega(x)\,dx=\int_0^{\infty}p\lambda^{p-1}\omega\big(\big\{x\in E:|f(x)|>\lambda\big\}\big)\,d\lambda.
\end{equation}
\end{lem}

\begin{proof}
Let $\omega$ be a weight function on $\mathbb R^n$. For $1\leq p<\infty$, we can deduce that
\begin{equation*}
\begin{split}
\int_{E}|f(x)|^p\omega(x)\,dx&=\int_{E}\bigg(\int_0^{|f(x)|}p\lambda^{p-1}\,d\lambda\bigg)\omega(x)\,dx\\
&=\int_{E}\bigg(\int_0^{\infty}p\lambda^{p-1}\chi_{[0,|f(x)|)}(\lambda)\,d\lambda\bigg)\omega(x)\,dx.
\end{split}
\end{equation*}
Changing the order of integration (the Fubini theorem) gives us that
\begin{equation*}
\begin{split}
\int_{E}|f(x)|^p\omega(x)\,dx&=\int_0^{\infty}p\lambda^{p-1}\bigg(\int_{E}\chi_{[0,|f(x)|)}(\lambda)\omega(x)\,dx\bigg)d\lambda\\
&=\int_0^{\infty}p\lambda^{p-1}\bigg(\int_{\{x\in E:|f(x)|>\lambda\}}\omega(x)\,dx\bigg)d\lambda\\
&=\int_0^{\infty}p\lambda^{p-1}\omega\big(\big\{x\in E:|f(x)|>\lambda\big\}\big)\,d\lambda.
\end{split}
\end{equation*}
This proves \eqref{layer}. Recall that for a measurable function $f$ defined on $E\subseteq\mathbb R^n$, the function $d_{f}(\cdot;\omega)$ defined on $[0,\infty)$ is called the $\omega$-distribution function of $f$, where
\begin{equation*}
d_{f}(\lambda;\omega):=\omega\big(\big\{x\in E:|f(x)|>\lambda\big\}\big).
\end{equation*}
By the formula \eqref{layer}, we can evaluate the $L^p$-norm (with respect to $\omega$) of a function $f$ in terms of the $\omega$-distribution function.
\end{proof}

\begin{lem}\label{comparelem}
Let $n\in\mathbb N$, $1\leq p<\infty$ and $\nu\in A_p$. Then for any given cube $\mathcal{Q}\subset\mathbb R^n$, there exists a positive number $\delta>0$ such that
\begin{equation}\label{comp}
\frac{\nu(E)}{\nu(\mathcal{Q})}\leq C^{*}\left(\frac{m(E)}{m(\mathcal{Q})}\right)^{\delta}
\end{equation}
holds for any measurable subset $E$ of $\mathcal{Q}$, where $C^{*}>0$ is a (universal) constant which does not depend on $E$ and $\mathcal{Q}$.
\end{lem}

\begin{proof}
Since $\nu\in A_p$ with $1\leq p<\infty$, we know that there exist constants $C$ and $\varepsilon>0$, depending only on $p$ and the $A_p$ constant of $\nu$, such that for any cube $\mathcal{Q}\subset\mathbb R^n$, the following reverse H\"older inequality holds (see \cite{duoand} and \cite{grafakos2}).
\begin{equation}\label{reverseh}
\bigg(\frac{1}{m(\mathcal{Q})}\int_{\mathcal{Q}}\nu(x)^{1+\varepsilon}\,dx\bigg)^{\frac{1}{1+\varepsilon}}
\leq C^{*}\bigg(\frac{1}{m(\mathcal{Q})}\int_{\mathcal{Q}}\nu(x)\,dx\bigg).
\end{equation}
Actually, there is a sharp version of \eqref{reverseh}. More precisely, by using the sharp reverse H\"older inequality for $A_p$ weights obtained recently in \cite[Lemma 2.3]{chung}, we have that
\begin{equation*}
C^{*}:=2\quad \& \quad \varepsilon:=\frac{1}{2^{2p+1+n}[\nu]_{A_p}}.
\end{equation*}
This result is due to P\'{e}rez (see also \cite[Lemma 3.26]{benyi}). Suppose that $E\subseteq Q$. By H\"older's inequality with exponent $1+\varepsilon$ and \eqref{reverseh}, we can deduce that
\begin{equation*}
\begin{split}
\nu(E)&=\int_{\mathcal{Q}}\chi_E(x)\cdot\nu(x)\,dx\\
&\leq\bigg(\int_{\mathcal{Q}}\chi_E(x)^{\frac{1+\varepsilon}{\varepsilon}}\,dx\bigg)^{\frac{\varepsilon}{1+\varepsilon}}
\bigg(\int_{\mathcal{Q}}\nu(x)^{1+\varepsilon}\,dx\bigg)^{\frac{1}{1+\varepsilon}}\\
&\leq 2m(E)^{\frac{\varepsilon}{1+\varepsilon}}
m(\mathcal{Q})^{\frac{1}{1+\varepsilon}}\bigg(\frac{1}{m(\mathcal{Q})}\int_{\mathcal{Q}}\nu(x)\,dx\bigg)\\
&=2\left(\frac{m(E)}{m(\mathcal{Q})}\right)^{\frac{\varepsilon}{1+\varepsilon}}\nu(\mathcal{Q}).
\end{split}
\end{equation*}
This gives \eqref{comp} with $\delta=\varepsilon/{(1+\varepsilon)}$ and $\varepsilon=\frac{1}{2^{2p+1+n}[\nu]_{A_p}}$.
\end{proof}

\begin{rem}
It was shown in \cite{hy1} and \cite{hy2} that one can get a sharper reverse H\"{o}lder exponent by using the Fujii--Wilson $A_{\infty}$-constant $[\omega]_{A_{\infty}}$, where $[\omega]_{A_{\infty}}$ is defined as ($M$ stands for the standard Hardy--Littlewood maximal operator)
\begin{equation*}
[\omega]_{A_{\infty}}:=\sup_{Q\subset\mathbb R^n}\frac{1}{\omega(Q)}\int_{Q}M(\omega\chi_{Q})(x)\,dx.
\end{equation*}
By using the results from \cite{hy1} and \cite{hy2}, the estimate \eqref{comp} of Lemma \ref{comparelem} can be further improved. Here we omit the details.
\end{rem}

\begin{lem}\label{expblo}
If $f\in \mathrm{BLO}(\omega)$ with $\omega\in A_1$, then there exist two positive constants $C_1$ and $C_2$ such that for every cube $\mathcal{Q}=Q(x_0,r)$ and every $\lambda>0$,
\begin{equation}\label{maincw}
\begin{split}
&m\Big(\Big\{x\in \mathcal{Q}:\Big[f(x)-\underset{y\in\mathcal{Q}}{\mathrm{ess\,inf}}\,f(y)\Big]>\lambda\Big\}\Big)\\
&\leq {C}_1\cdot m(\mathcal{Q})\exp\bigg\{-\Big[[\omega]_{A_1}\underset{x\in\mathcal{Q}}{\mathrm{ess\,inf}}\,\omega(x)\Big]^{-1}
\frac{{C}_2\lambda}{\|f\|_{\mathrm{BLO}(\omega)}}\bigg\}.
\end{split}
\end{equation}
More specifically, we may choose
\begin{equation*}
{C}_1=e\quad and \quad {C}_2=\frac{1}{2^ne}.
\end{equation*}
\end{lem}

\begin{proof}[Proof of Lemma $\ref{expblo}$]
Some ideas of the proof of this lemma come from \cite{duoand} and \cite{grafakos2}. The proof has five main steps.

\textbf{Step 1}. Without loss of generality, we may assume that $\|f\|_{\mathrm{BLO}(\omega)}=1$ with $\omega\in A_1$. Note that if $\lambda\leq[\omega]_{A_1}\underset{x\in\mathcal{Q}}{\mathrm{ess\,inf}}\,\omega(x)$, then the inequality \eqref{maincw} holds true by choosing ${C}_1=e$ and ${C}_2=1$. Now we suppose that $\lambda>[\omega]_{A_1}\underset{x\in\mathcal{Q}}{\mathrm{ess\,inf}}\,\omega(x)$. Then for each fixed cube $\mathcal{Q}=Q(x_0,r)$, we can apply the Calder\'{o}n--Zygmund decomposition to the function $f(x)-\underset{y\in\mathcal{Q}}{\mathrm{ess\,inf}}f(y)$ inside the cube $\mathcal{Q}$. Let $\sigma>1$ be a positive constant to be fixed below.
Since
\begin{equation*}
\bigg(\frac{1}{\omega(\mathcal{Q})}\int_{\mathcal{Q}}\Big[f(x)-\underset{y\in\mathcal{Q}}{\mathrm{ess\,inf}}\,f(y)\Big]\,dx\bigg)
\leq\|f\|_{\mathrm{BLO}(\omega)}=1<\sigma,
\end{equation*}
from the condition $\omega\in A_1$, it then follows that
\begin{equation*}
\begin{split}
&\bigg(\frac{1}{m(\mathcal{Q})}\int_{\mathcal{Q}}\Big[f(x)-\underset{y\in\mathcal{Q}}{\mathrm{ess\,inf}}\,f(y)\Big]\,dx\bigg)\\
&=\frac{\omega(\mathcal{Q})}{m(\mathcal{Q})}\cdot
\bigg(\frac{1}{\omega(\mathcal{Q})}\int_{\mathcal{Q}}\Big[f(x)-\underset{y\in\mathcal{Q}}{\mathrm{ess\,inf}}\,f(y)\Big]\,dx\bigg)
<[\omega]_{A_1}\underset{x\in\mathcal{Q}}{\mathrm{ess\,inf}}\,\omega(x)\cdot\sigma.
\end{split}
\end{equation*}
We set $\mathcal{A}_{\omega}=[\omega]_{A_1}\underset{x\in\mathcal{Q}}{\mathrm{ess\,inf}}\,\omega(x)$ and then follow the same argument (the so-called stopping time argument) as in the proof of \cite[Theorem 7.1.6]{grafakos2} to obtain a collection of (pairwise disjoint) cubes $\big\{Q^{(1)}_j\big\}_j$ satisfying the following properties:
\begin{equation*}
\begin{split}
&(A)\mbox{-1}. ~~\mbox{The interior of every cube}~ Q^{(1)}_j ~\mbox{is contained in}~ \mathcal{Q};\\
&(B)\mbox{-1}. ~~\mathcal{A}_{\omega}\cdot\sigma
<\frac{1}{m(Q^{(1)}_j)}\int_{Q^{(1)}_j}\Big[f(x)-\underset{y\in\mathcal{Q}}{\mathrm{ess\,inf}}\,f(y)\Big]\,dx
\leq \mathcal{A}_{\omega}\cdot2^n\sigma;\\
&(C)\mbox{-1}. ~~0\leq\underset{y\in{Q^{(1)}_j}}{\mathrm{ess\,inf}}\,f(y)-\underset{y\in\mathcal{Q}}{\mathrm{ess\,inf}}\,f(y)
\leq \mathcal{A}_{\omega}\cdot2^n\sigma;\\
&(D)\mbox{-1}. ~~\sum_{j}m\big(Q^{(1)}_j\big)\leq\frac{m(\mathcal{Q})}{\sigma};\\
&(E)\mbox{-1}. ~~f(x)-\underset{y\in\mathcal{Q}}{\mathrm{ess\,inf}}\,f(y)
\leq \mathcal{A}_{\omega}\cdot\sigma,~~a.e.~x\in \mathcal{Q}\setminus\bigcup_jQ^{(1)}_j.
\end{split}
\end{equation*}
We prove these properties $(A)$-1 through $(E)$-1. Obviously, properties $(A)$-1 and $(B)$-1 hold by the selection criterion of the cubes $Q^{(1)}_j$(can be viewed as the first generation of $\mathcal{Q}$). Since $Q^{(1)}_j\subset \mathcal{Q}$ and
\begin{equation*}
\begin{split}
\underset{y\in{Q^{(1)}_j}}{\mathrm{ess\,inf}}\,f(y)
&=\frac{1}{m(Q^{(1)}_j)}\int_{Q^{(1)}_j}\underset{y\in{Q^{(1)}_j}}{\mathrm{ess\,inf}}\,f(y)\,dx\\
&\leq\frac{1}{m(Q^{(1)}_j)}\int_{Q^{(1)}_j}f(x)\,dx=f_{Q^{(1)}_j},
\end{split}
\end{equation*}
we get
\begin{equation*}
\begin{split}
0&\leq\underset{y\in{Q^{(1)}_j}}{\mathrm{ess\,inf}}\,f(y)-\underset{y\in\mathcal{Q}}{\mathrm{ess\,inf}}\,f(y)\\
&\leq\frac{1}{m(Q^{(1)}_j)}\int_{Q^{(1)}_j}f(x)\,dx-\underset{y\in\mathcal{Q}}{\mathrm{ess\,inf}}\,f(y)\\
&=\frac{1}{m(Q^{(1)}_j)}\int_{Q^{(1)}_j}\Big[f(x)-\underset{y\in\mathcal{Q}}{\mathrm{ess\,inf}}\,f(y)\Big]\,dx\\
&\leq \mathcal{A}_{\omega}\cdot 2^n\sigma,
\end{split}
\end{equation*}
where in the last inequality we have used $(B)$-1. Because the cubes $Q^{(1)}_j$ are pairwise disjoint, then it follows from $(B)$-1 that
\begin{equation*}
\begin{split}
\mathcal{A}_{\omega}\cdot\sum_{j}m\big(Q^{(1)}_j\big)&<\frac{1}{\sigma}
\sum_{j}\int_{Q^{(1)}_j}\Big[f(x)-\underset{y\in\mathcal{Q}}{\mathrm{ess\,inf}}\,f(y)\Big]\,dx\\
&=\frac{1}{\sigma}\int_{\bigcup_jQ^{(1)}_j}\Big[f(x)-\underset{y\in\mathcal{Q}}{\mathrm{ess\,inf}}\,f(y)\Big]\,dx\\
&\leq\frac{1}{\sigma}\int_{\mathcal{Q}}\Big[f(x)-\underset{y\in\mathcal{Q}}{\mathrm{ess\,inf}}\,f(y)\Big]\,dx\\
&\leq\frac{m(\mathcal{Q})}{\sigma}\cdot\mathcal{A}_{\omega}.
\end{split}
\end{equation*}
This is equivalent to $(D)$-1. $(E)$-1 is a consequence of the Lebesgue differentiation theorem.

\textbf{Step 2}. We now fix a selected cube $Q^{(1)}_{j'}$(first generation) and apply the same Calder\'{o}n--Zygmund decomposition to the function $f(x)-\underset{y\in{Q}^{(1)}_{j'}}{\mathrm{ess\,inf}}f(y)$ inside the cube $Q^{(1)}_{j'}$. Also repeat this process for any other cube of the first generation. Let $Q^{(1)}_{j'}=Q(x_1,r_1)$ be the cube centered at $x_1$ and with side length $r_1$. Observe that
\begin{equation*}
\bigg(\frac{1}{\omega(Q^{(1)}_{j'})}\int_{Q^{(1)}_{j'}}\Big[f(x)-\underset{y\in Q^{(1)}_{j'}}{\mathrm{ess\,inf}}\,f(y)\Big]\,dx\bigg)
\leq\|f\|_{\mathrm{BLO}(\omega)}=1<\sigma.
\end{equation*}
Arguing as in Step 1, we obtain a collection of (pairwise disjoint) cubes $\big\{Q^{(2)}_j\big\}_j$ satisfying the following properties:
\begin{equation*}
\begin{split}
&(A)\mbox{-2}. ~~\mbox{The interior of every cube}~ Q^{(2)}_j ~\mbox{is contained in a unique cube}~ {Q}^{(1)}_{j'};\\
&(B)\mbox{-2}. ~~\mathcal{A}_{\omega}\cdot\sigma
<\frac{1}{m(Q^{(2)}_j)}\int_{Q^{(2)}_j}\Big[f(x)-\underset{y\in{Q}^{(1)}_{j'}}{\mathrm{ess\,inf}}\,f(y)\Big]\,dx
\leq \mathcal{A}_{\omega}\cdot2^n\sigma;\\
&(C)\mbox{-2}. ~~0\leq\underset{y\in{Q}^{(2)}_{j}}{\mathrm{ess\,inf}}\,f(y)-\underset{y\in{Q^{(1)}_{j'}}}{\mathrm{ess\,inf}}\,f(y)
\leq \mathcal{A}_{\omega}\cdot2^n\sigma;\\
&(D)\mbox{-2}. ~~\sum_{j}m\big(Q^{(2)}_j\big)\leq\frac{1}{\sigma}\sum_{j'} m\big(Q^{(1)}_{j'}\big);\\
&(E)\mbox{-2}. ~~f(x)-\underset{y\in{Q}^{(1)}_{j'}}{\mathrm{ess\,inf}}\,f(y)
\leq\mathcal{A}_{\omega}\cdot\sigma,~~a.e.~x\in {Q}^{(1)}_{j'}\setminus\bigcup_jQ^{(2)}_j.
\end{split}
\end{equation*}
In fact, it is clear that properties $(A)$-2 and $(B)$-2 hold by the selection criterion of the cubes $Q^{(2)}_j$(can be viewed as the second generation of $\mathcal{Q}$). Since $Q^{(2)}_j\subset Q^{(1)}_{j'}$ and
\begin{equation*}
\begin{split}
\underset{y\in{Q^{(2)}_j}}{\mathrm{ess\,inf}}\,f(y)
&=\frac{1}{m(Q^{(2)}_j)}\int_{Q^{(2)}_j}\underset{y\in{Q^{(2)}_j}}{\mathrm{ess\,inf}}\,f(y)\,dx
\leq\frac{1}{m(Q^{(2)}_j)}\int_{Q^{(2)}_j}f(x)\,dx,
\end{split}
\end{equation*}
so we have
\begin{equation*}
\begin{split}
0\leq\underset{y\in{Q}^{(2)}_{j}}{\mathrm{ess\,inf}}\,f(y)-\underset{y\in{Q^{(1)}_{j'}}}{\mathrm{ess\,inf}}\,f(y)
&\leq\frac{1}{m(Q^{(2)}_j)}\int_{Q^{(2)}_j}\Big[f(x)-\underset{y\in{Q}^{(1)}_{j'}}{\mathrm{ess\,inf}}\,f(y)\Big]\,dx\\
&\leq \mathcal{A}_{\omega}\cdot2^n\sigma,
\end{split}
\end{equation*}
due to property $(B)$-2. By the Lebesgue differentiation theorem, $(E)$-2 holds. It remains only to study the last property $(D)$-2. Notice that the cubes $Q^{(2)}_j$ are also pairwise disjoint and each selected cube $Q^{(2)}_j$ is contained in a unique cube $Q^{(1)}_{j'}$, we can deduce that
\begin{equation*}
\begin{split}
\mathcal{A}_{\omega}\cdot\sum_{j}\big|Q^{(2)}_j\big|
&<\frac{1}{\sigma}\sum_{j}\int_{Q^{(2)}_j}\Big[f(x)-\underset{y\in{Q}^{(1)}_{j'}}{\mathrm{ess\,inf}}\,f(y)\Big]\,dx\\
&\leq\frac{1}{\sigma}\sum_{j'}\int_{Q^{(1)}_{j'}}\Big[f(x)-\underset{y\in{Q}^{(1)}_{j'}}{\mathrm{ess\,inf}}\,f(y)\Big]\,dx\\
&\leq\frac{1}{\sigma}\sum_{j'}m\big(Q^{(1)}_{j'}\big)
\cdot\frac{\omega(Q^{(1)}_{j'})}{m(Q^{(1)}_{j'})}\big\|f\big\|_{\mathrm{BLO}(\omega)}\\
&\leq\frac{1}{\sigma}\sum_{j'}m\big(Q^{(1)}_{j'}\big)\cdot\mathcal{A}_{\omega}.
\end{split}
\end{equation*}
This is just the desired estimate. Summarizing the estimates derived above($(E)$-2 and $(C)$-1), we can deduce that for almost every $x\in {Q}^{(1)}_{j'}\setminus\bigcup_jQ^{(2)}_j$,
\begin{equation*}
\begin{split}
f(x)-\underset{y\in\mathcal{Q}}{\mathrm{ess\,inf}}\,f(y)
&=f(x)-\underset{y\in{Q}^{(1)}_{j'}}{\mathrm{ess\,inf}}\,f(y)+\underset{y\in{Q}^{(1)}_{j'}}{\mathrm{ess\,inf}}\,f(y)
-\underset{y\in\mathcal{Q}}{\mathrm{ess\,inf}}\,f(y)\\
&\leq\sigma\cdot\mathcal{A}_{\omega}+2^n\sigma\cdot\mathcal{A}_{\omega},~~a.e.~x\in {Q}^{(1)}_{j'}\setminus\bigcup_jQ^{(2)}_j.
\end{split}
\end{equation*}
This estimate, together with $(E)$-1, implies that for almost every $x\in {\mathcal{Q}}\setminus\bigcup_jQ^{(2)}_j$,
\begin{equation*}
\begin{split}
f(x)-\underset{y\in\mathcal{Q}}{\mathrm{ess\,inf}}\,f(y)
&\leq \sigma\cdot\mathcal{A}_{\omega}+2^n\sigma\cdot\mathcal{A}_{\omega}\\
&\leq2\sigma\cdot 2^{n}\mathcal{A}_{\omega}.
\end{split}
\end{equation*}
Moreover, from $(D)$-1 and $(D)$-2, we conclude that
\begin{equation*}
\sum_{j}m\big(Q^{(2)}_j\big)\leq\frac{1}{\sigma}\sum_{j'}m\big(Q^{(1)}_{j'}\big)\leq\frac{m(\mathcal{Q})}{\sigma^2}.
\end{equation*}

\textbf{Step 3}. We repeat this process indefinitely to obtain a collection of cubes $\big\{Q^{(k)}_j\big\}_j$ satisfying the following properties:
\begin{equation*}
\begin{split}
&(A)\mbox{-k}. ~~\mbox{The interior of every cube}~ Q^{(k)}_j ~\mbox{is contained in a unique cube}~ {Q}^{(k-1)}_{j'};\\
&(B)\mbox{-k}. ~~\mathcal{A}_{\omega}\cdot\sigma
<\frac{1}{m(Q^{(k)}_j)}\int_{Q^{(k)}_j}\Big[f(x)-\underset{y\in{Q}^{(k-1)}_{j'}}{\mathrm{ess\,inf}}\,f(y)\Big]\,dx
\leq \mathcal{A}_{\omega}\cdot2^n\sigma;\\
&(C)\mbox{-k}. ~~0\leq\underset{y\in{Q}^{(k)}_{j}}{\mathrm{ess\,inf}}\,f(y)-\underset{y\in{Q^{(k-1)}_{j'}}}{\mathrm{ess\,inf}}\,f(y)
\leq \mathcal{A}_{\omega}\cdot2^n\sigma;\\
&(D)\mbox{-k}. ~~\sum_{j}m\big(Q^{(k)}_j\big)\leq\frac{1}{\sigma}\sum_{j'}m\big(Q^{(k-1)}_{j'}\big);\\
&(E)\mbox{-k}. ~~f(x)-\underset{y\in{Q}^{(k-1)}_{j'}}{\mathrm{ess\,inf}}\,f(y)
\leq\mathcal{A}_{\omega}\cdot\sigma,~~a.e.~x\in {Q}^{(k-1)}_{j'}\setminus\bigcup_jQ^{(k)}_j.
\end{split}
\end{equation*}
Here ${Q}^{(k-1)}_{j'}$ denotes the cube centered at $x_{k-1}$ with side length $r_{k-1}$. By induction, from the previous proof, it actually follows that
\begin{equation*}
f(x)-\underset{y\in\mathcal{Q}}{\mathrm{ess\,inf}}\,f(y)
\leq k\sigma\cdot 2^n\mathcal{A}_{\omega},~~a.e.~x\in \mathcal{Q}\setminus\bigcup_{\ell}Q^{(k)}_{\ell},
\end{equation*}
and
\begin{equation}\label{www}
\sum_{\ell}m\big(Q^{(k)}_{\ell}\big)\leq\frac{m(\mathcal{Q})}{\sigma^{k}},\quad k=1,2,3,\dots.
\end{equation}
Therefore
\begin{equation}\label{hhh}
\Big\{x\in \mathcal{Q}:\Big[f(x)-\underset{y\in\mathcal{Q}}{\mathrm{ess\,inf}}\,f(y)\Big]>k\sigma \cdot2^n\mathcal{A}_{\omega}\Big\}
\subseteq\bigcup_{\ell}Q^{(k)}_{\ell},k=1,2,3,\dots.
\end{equation}

\textbf{Step 4}. Since
\begin{equation*}
(0,+\infty)=\bigcup_{k=0}^{\infty}
\bigg(k\sigma2^n\mathcal{A}_{\omega},(k+1)\sigma2^n\mathcal{A}_{\omega}\bigg],
\end{equation*}
then for each fixed $\lambda\in(0,+\infty)$, we can write
\begin{equation*}
k\sigma2^n\mathcal{A}_{\omega}<\lambda\leq (k+1)\sigma2^n\mathcal{A}_{\omega}
\end{equation*}
for some $k\in \mathbb{N}\cup\{0\}$, and hence
\begin{equation*}
\begin{split}
&m\Big(\Big\{x\in \mathcal{Q}:\Big[f(x)-\underset{y\in\mathcal{Q}}{\mathrm{ess\,inf}}\,f(y)\Big]>\lambda\Big\}\Big)\\
&\leq m\Big(\Big\{x\in \mathcal{Q}:\Big[f(x)-\underset{y\in\mathcal{Q}}{\mathrm{ess\,inf}}\,f(y)\Big]>k\sigma 2^n\mathcal{A}_{\omega}\Big\}\Big)\\
&\leq\sum_{\ell}m\big(Q_{\ell}^{(k)}\big)\leq\frac{m(\mathcal{Q})}{\sigma^k}\\
&=\sigma\cdot m(\mathcal{Q})\cdot\frac{\exp\{-k\log\sigma\}}{\sigma},
\end{split}
\end{equation*}
where in the last two inequalities we have used \eqref{www} and \eqref{hhh}, respectively. Now choose $\sigma=e>1$, we then have
\begin{equation*}
\begin{split}
&m\Big(\Big\{x\in \mathcal{Q}:\Big[f(x)-\underset{y\in\mathcal{Q}}{\mathrm{ess\,inf}}\,f(y)\Big]>\lambda\Big\}\Big)\\
&\leq e\cdot m(\mathcal{Q})\exp\big\{-(k+1)\big\}\\
&\leq e\cdot m(\mathcal{Q})\exp\Big\{-\frac{\lambda}{\mathcal{A}_{\omega}\cdot2^ne}\Big\}.
\end{split}
\end{equation*}
This concludes the proof of Lemma \ref{expblo} for the special case that $f\in\mathrm{BLO}(\omega)$ with $\|f\|_{\mathrm{BLO}(\omega)}=1$.

\textbf{Step 5}. We now proceed to the general case. In order to do so, we set
\begin{equation*}
\widetilde{f}(x):=\frac{f(x)}{\|f\|_{\mathrm{BLO}(\omega)}}.
\end{equation*}
By the definition of $\|\cdot\|_{\mathrm{BLO}(\omega)}$, we have
\begin{equation*}
\|\widetilde{f}\|_{\mathrm{BLO}(\omega)}=1\quad \&\quad
\underset{y\in\mathcal{Q}}{\mathrm{ess\,inf}}\,f(y)=\|f\|_{\mathrm{BLO}(\omega)}
\cdot\underset{y\in\mathcal{Q}}{\mathrm{ess\,inf}}\,\widetilde{f}(y).
\end{equation*}
Hence,
\begin{equation*}
\begin{split}
&m\Big(\Big\{x\in \mathcal{Q}:\Big[f(x)-\underset{y\in\mathcal{Q}}{\mathrm{ess\,inf}}\,f(y)\Big]>\lambda\Big\}\Big)\\
&=m\Big(\Big\{x\in \mathcal{Q}:\Big[\widetilde{f}(x)-\underset{y\in\mathcal{Q}}{\mathrm{ess\,inf}}\,\widetilde{f}(y)\Big]>
\frac{\lambda}{\|f\|_{\mathrm{BLO}(\omega)}}\Big\}\Big)\\
&\leq {C}_1 m(\mathcal{Q})\exp\bigg\{-\Big[[\omega]_{A_1}\underset{x\in\mathcal{Q}}{\mathrm{ess\,inf}}\,\omega(x)\Big]^{-1}
\frac{{C}_2\lambda}{\|f\|_{\mathrm{BLO}(\omega)}}\bigg\},
\end{split}
\end{equation*}
with precise constants
\begin{equation*}
{C}_1=e\quad \& \quad {C}_2=\frac{1}{2^ne}.
\end{equation*}
We are done.
\end{proof}

Based on Lemma \ref{expblo}, we can prove the following result.

\begin{thm}\label{weightedblomain}
For any $1\leq p<\infty$ and $\omega\in A_1$, define
\begin{equation*}
\big\|f\big\|_{\mathrm{BLO}^p(\omega)}:=\sup_{\mathcal{Q}\subset\mathbb R^n}
\bigg(\frac{1}{\omega(\mathcal{Q})}\int_{\mathcal{Q}}
\Big[f(x)-\underset{y\in\mathcal{Q}}{\mathrm{ess\,inf}}\,f(y)\Big]^p\omega(x)^{1-p}dx\bigg)^{1/p}.
\end{equation*}
When $p=1$, we simply write $\|\cdot\|_{\mathrm{BLO}^p(\omega)}=\|\cdot\|_{\mathrm{BLO}(\omega)}$. Then we have
\begin{equation*}
\big\|f\big\|_{\mathrm{BLO}^p(\omega)}\approx\big\|f\big\|_{\mathrm{BLO}(\omega)},
\end{equation*}
for each $1<p<\infty$. To be more precise,
\begin{equation*}
\big\|f\big\|_{\mathrm{BLO}(\omega)}\leq\big\|f\big\|_{\mathrm{BLO}^p(\omega)}
\leq[\omega]_{A_1}\Big(C^{*}\cdot p\Gamma(p)\Big)^{1/p}
\cdot\frac{{C}_1^{\delta/p}}{{C}_2\delta}\big\|f\big\|_{\mathrm{BLO}(\omega)},
\end{equation*}
where $C^{*}$ and $\delta$ are the constants appearing in Lemma $\ref{comparelem}$, and $\Gamma(\cdot)$ denotes the usual gamma function, which is given by
\begin{equation*}
\Gamma(x)=\int_0^{\infty}\mu^{x-1}e^{-\mu}\,d\mu.
\end{equation*}
\end{thm}

Now we define
\begin{equation*}
\mathrm{BLO}^p(\omega):=\Big\{f\in L^1_{\mathrm{loc}}(\mathbb R^n):
\big\|f\big\|_{\mathrm{BLO}^p(\omega)}<+\infty\Big\},\quad 1\leq p<\infty.
\end{equation*}
This result tells us that for all $1\leq p<\infty$, the spaces $\mathrm{BLO}^p(\omega)$ coincide, and the ``norms" $\|\cdot\|_{\mathrm{BLO}^p(\omega)}$ are mutually equivalent with respect to different values of $p$, provided that $\omega\in A_1$.

\begin{proof}[Proof of Theorem $\ref{weightedblomain}$]
(1) Suppose that $f\in \mathrm{BLO}^p(\omega)$ with $1<p<\infty$. By using H\"older's inequality and the fact that $1-p=-{p}/{p'}$, we obtain that for any cube $\mathcal{Q}\subset\mathbb R^n$,
\begin{align*}
&\frac{1}{\omega(\mathcal{Q})}\int_{\mathcal{Q}}\Big[f(x)-\underset{y\in\mathcal{Q}}{\mathrm{ess\,inf}}\,f(y)\Big]\,dx\nonumber\\
&=\frac{1}{\omega(\mathcal{Q})}\int_{\mathcal{Q}}\Big[f(x)-\underset{y\in\mathcal{Q}}{\mathrm{ess\,inf}}\,f(y)\Big]
\omega(x)^{-1/{p'}}\cdot\omega(x)^{1/{p'}}\,dx\nonumber\\
&\leq\frac{1}{\omega(\mathcal{Q})}\bigg(\int_{\mathcal{Q}}
\Big[f(x)-\underset{y\in\mathcal{Q}}{\mathrm{ess\,inf}}\,f(y)\Big]^p\omega(x)^{-{p}/{p'}}dx\bigg)^{1/p}
\bigg(\int_{\mathcal{Q}}\omega(x)\,dx\bigg)^{1/{p'}}\nonumber\\
&=\bigg(\frac{1}{\omega(\mathcal{Q})}\int_{\mathcal{Q}}
\Big[f(x)-\underset{y\in\mathcal{Q}}{\mathrm{ess\,inf}}\,f(y)\Big]^p\omega(x)^{1-p}\,dx\bigg)^{1/{p}}
\leq\big\|f\big\|_{\mathrm{BLO}^p(\omega)}.
\end{align*}
(2) Let $f\in \mathrm{BLO}(\omega)$ with $\omega\in A_1$. According to Lemma \ref{expblo}, there are two constants ${C}_1,{C}_2>0$ such that for any $\lambda>0$ and for any cube $\mathcal{Q}\subset\mathbb R^n$,
\begin{equation*}
\begin{split}
&m\Big(\Big\{x\in \mathcal{Q}:\Big[f(x)-\underset{y\in\mathcal{Q}}{\mathrm{ess\,inf}}\,f(y)\Big]>\lambda\Big\}\Big)\\
&\leq{C}_1 m(\mathcal{Q})\exp\bigg\{-\Big[[\omega]_{A_1}\underset{x\in\mathcal{Q}}{\mathrm{ess\,inf}}\,\omega(x)\Big]^{-1}
\frac{{C}_2\lambda}{\|f\|_{\mathrm{BLO}(\omega)}}\bigg\}.
\end{split}
\end{equation*}
By the property of $A_p$ weights, we know that for $1<p<\infty$,
\begin{equation*}
\omega\in A_{p'}\Longleftrightarrow \omega^{1-p}\in A_p.
\end{equation*}
If we set $\nu(x):=\omega(x)^{1-p}$ with $1<p<\infty$, then $\nu\in A_p$ since $\omega\in A_1\subset A_{p'}$ for $1<p<\infty$. By using Lemma \ref{comparelem}, we get
\begin{equation*}
\begin{split}
&\nu\Big(\Big\{x\in \mathcal{Q}:\Big[f(x)-\underset{y\in\mathcal{Q}}{\mathrm{ess\,inf}}\,f(y)\Big]>\lambda\Big\}\Big)\\
&\leq C^{*}\cdot{C}_1^{\delta}\nu(\mathcal{Q})
\exp\bigg\{-\Big[[\omega]_{A_1}\underset{x\in\mathcal{Q}}{\mathrm{ess\,inf}}\,\omega(x)\Big]^{-1}
\frac{{C}_2\delta\lambda}{\|f\|_{\mathrm{BLO}(\omega)}}\bigg\}.
\end{split}
\end{equation*}
This estimate, together with Lemma \ref{lem71}, gives us that for any $1<p<\infty$
\begin{align*}
&\bigg(\frac{1}{\omega(\mathcal{Q})}\int_{\mathcal{Q}}\Big[f(x)-\underset{y\in\mathcal{Q}}{\mathrm{ess\,inf}}\,f(y)\Big]^p
\nu(x)\,dx\bigg)^{1/p}\nonumber\\
&=\bigg(\frac{1}{\omega(\mathcal{Q})}\int_0^{\infty}p\lambda^{p-1}\nu\Big(\Big\{x\in \mathcal{Q}:\Big[f(x)-\underset{y\in\mathcal{Q}}{\mathrm{ess\,inf}}\,f(y)\Big]>\lambda\Big\}\Big)d\lambda\bigg)^{1/p}\nonumber\\
&\leq\bigg(C^{*}\cdot{C}_1^{\delta}\int_0^{\infty}p\lambda^{p-1}
\exp\bigg\{-\Big[[\omega]_{A_1}\underset{x\in\mathcal{Q}}{\mathrm{ess\,inf}}\,\omega(x)\Big]^{-1}
\frac{{C}_2\delta\lambda}{\|f\|_{\mathrm{BLO}(\omega)}}\bigg\}d\lambda\bigg)^{1/p}\\
&\times\left(\frac{\nu(\mathcal{Q})}{\omega(\mathcal{Q})}\right)^{1/p}.
\end{align*}
Observe that
\begin{equation*}
\nu(\mathcal{Q})^{1/p}=\bigg\{\int_{\mathcal{Q}}\frac{\omega(x)}{\omega(x)^p}dx\bigg\}^{1/p}\leq
\omega(\mathcal{Q})^{1/p}\cdot\Big[\underset{x\in\mathcal{Q}}{\mathrm{ess\,inf}}\,\omega(x)\Big]^{-1}.
\end{equation*}
This implies that
\begin{equation*}
\left(\frac{\nu(\mathcal{Q})}{\omega(\mathcal{Q})}\right)^{1/p}\leq
\Big[\underset{x\in\mathcal{Q}}{\mathrm{ess\,inf}}\,\omega(x)\Big]^{-1}.
\end{equation*}
Moreover, by making the substitution
\begin{equation*}
\mu:=\Big[[\omega]_{A_1}\underset{x\in\mathcal{Q}}{\mathrm{ess\,inf}}\,\omega(x)\Big]^{-1}
\frac{{C}_2\delta\lambda}{\|f\|_{\mathrm{BLO}(\omega)}},
\end{equation*}
we can deduce that for any cube $\mathcal{Q}$ in $\mathbb R^n$,
\begin{align*}
&\bigg(\frac{1}{\omega(\mathcal{Q})}\int_{\mathcal{Q}}\Big[f(x)-\underset{y\in\mathcal{Q}}{\mathrm{ess\,inf}}\,f(y)\Big]^p
\nu(x)\,dx\bigg)^{1/p}\nonumber\\
&\leq\Big(C^{*}\cdot{C}_1^{\delta}p\Big)^{1/p}
\Big[[\omega]_{A_1}\underset{x\in\mathcal{Q}}{\mathrm{ess\,inf}}\,\omega(x)\Big]
\frac{\|f\|_{\mathrm{BLO}(\omega)}}{{C}_2\delta}
\times\bigg(\int_0^{\infty}\mu^{p-1}e^{-\mu}\,d\mu\bigg)^{1/p}
\left(\frac{\nu(\mathcal{Q})}{\omega(\mathcal{Q})}\right)^{1/p}\nonumber\\
&\leq[\omega]_{A_1}\Big(C^{*}\cdot p\Gamma(p)\Big)^{1/p}
\cdot\frac{{C}_1^{\delta/p}}{{C}_2\delta}\big\|f\big\|_{\mathrm{BLO}(\omega)}.
\end{align*}
This gives the desired inequality. Collecting the above estimates, we conclude the proof of Theorem \ref{weightedblomain}.
\end{proof}

Let $f\in \mathrm{BMO}(\omega)$ and $\omega\in A_1$. Arguing as in the proof of Lemma \ref{expblo}, we can also prove that for every cube $\mathcal{Q}=Q(x_0,r)$ and every $\lambda>0$,
\begin{equation}\label{jlbmo}
\begin{split}
&m\Big(\Big\{x\in \mathcal{Q}:\big|f(x)-f_{\mathcal{Q}}\big|>\lambda\Big\}\Big)\\
&\leq e\cdot m(\mathcal{Q})\exp\bigg\{-\mathcal{A}_{\omega}^{-1}\cdot\frac{\lambda}{2^ne\|f\|_{\mathrm{BMO}(\omega)}}\bigg\},
\end{split}
\end{equation}
where
\begin{equation*}
\mathcal{A}_{\omega}:=[\omega]_{A_1}\underset{x\in\mathcal{Q}}{\mathrm{ess\,inf}}\,\omega(x).
\end{equation*}

Based on this, then we get
\begin{thm}\label{lastthm2}
For any $1\leq p<\infty$ and $\omega\in A_1$, define
\begin{equation*}
\big\|f\big\|_{\mathrm{BMO}^p(\omega)}:=\sup_{\mathcal{Q}\subset\mathbb R^n}
\bigg(\frac{1}{\omega(\mathcal{Q})}\int_{\mathcal{Q}}\big|f(x)-f_{\mathcal{Q}}\big|^p\omega(x)^{1-p}\,dx\bigg)^{1/p}.
\end{equation*}
When $p=1$, we simply write $\|\cdot\|_{\mathrm{BMO}^p(\omega)}=\|\cdot\|_{\mathrm{BMO}(\omega)}$. Then we have
\begin{equation*}
\big\|f\big\|_{\mathrm{BMO}^p(\omega)}\approx\big\|f\big\|_{\mathrm{BMO}(\omega)},
\end{equation*}
for each $1<p<\infty$.
\end{thm}
As in the proof of Theorem \ref{weightedblomain}, we can also show that
\begin{equation*}
\big\|f\big\|_{\mathrm{BMO}(\omega)}\leq\big\|f\big\|_{\mathrm{BMO}^p(\omega)}
\leq[\omega]_{A_1}\Big(C^{*}\cdot p\Gamma(p)\Big)^{1/p}
\cdot\frac{{C}_1^{\delta/p}}{{C}_2\delta}\big\|f\big\|_{\mathrm{BMO}(\omega)},
\end{equation*}
where $C^{*}$ and $\delta$ are the same as in Lemma $\ref{comparelem}$. Indeed, we obtain
\begin{equation*}
\frac{1}{\omega(\mathcal{Q})}\int_{\mathcal{Q}}\big|f(x)-f_{\mathcal{Q}}\big|\,dx
\leq\bigg(\frac{1}{\omega(\mathcal{Q})}\int_{\mathcal{Q}}\big|f(x)-f_{\mathcal{Q}}\big|^p\omega(x)^{1-p}dx\bigg)^{1/p}
\end{equation*}
by H\"{o}lder's inequality, and
\begin{equation*}
\bigg(\frac{1}{\omega(\mathcal{Q})}\int_{\mathcal{Q}}\big|f(x)-f_{\mathcal{Q}}\big|^p\omega(x)^{1-p}dx\bigg)^{1/p}
\leq[\omega]_{A_1}\Big(C^{*}\cdot p\Gamma(p)\Big)^{1/p}
\cdot\frac{{C}_1^{\delta/p}}{{C}_2\delta}\big\|f\big\|_{\mathrm{BMO}(\omega)}
\end{equation*}
by the inequality \eqref{jlbmo}, Lemmas \ref{lem71} and \ref{comparelem}. Hence, the desired result follows.

Similarly, we define
\begin{equation*}
\mathrm{BMO}^p(\omega):=\Big\{f\in L^1_{\mathrm{loc}}(\mathbb R^n):
\big\|f\big\|_{\mathrm{BMO}^p(\omega)}<+\infty\Big\},\quad 1\leq p<\infty.
\end{equation*}
Theorem 7.6 tells us that for all $1\leq p<\infty$, the spaces $\mathrm{BMO}^p(\omega)$ coincide, and the norms $\|\cdot\|_{\mathrm{BMO}^p(\omega)}$ are mutually equivalent with respect to different values of $p$, provided that $\omega\in A_1$.

As an immediate consequence of Theorem \ref{mainthm1} through Theorem \ref{mainthm3}, we have the following results.

\begin{cor}
For any $1<p<\infty$, ${f}\in \mathrm{BMO}^p(\omega)$ and $\omega\in A_1$, then $\mathcal{G}({f})$(or $\mathcal{S}({f})$) is either infinite everywhere or finite almost everywhere, and in the latter case, we have
\begin{equation*}
\big\|\mathcal{G}({f})\big\|_{\mathrm{BLO}^p(\omega)}\lesssim\big\|f\big\|_{\mathrm{BMO}^p(\omega)}
\quad \mbox{and} \quad \big\|\mathcal{S}({f})\big\|_{\mathrm{BLO}^p(\omega)}\lesssim\big\|f\big\|_{\mathrm{BMO}^p(\omega)}.
\end{equation*}
\end{cor}

\begin{cor}
Suppose that $(\lambda-3)n>2\delta+2\gamma$. For any $1<p<\infty$, ${f}\in \mathrm{BMO}^p(\omega)$ and $\omega\in A_1$, then $\mathcal{G}^{\ast}_{\lambda}({f})$ is either infinite everywhere or finite almost everywhere, and in the latter case, we have
\begin{equation*}
\big\|\mathcal{G}^{\ast}_{\lambda}({f})\big\|_{\mathrm{BLO}^p(\omega)}\lesssim\big\|f\big\|_{\mathrm{BMO}^p(\omega)}.
\end{equation*}
\end{cor}
Here the implicit constant is independent of ${f}$.

\section*{Acknowledgment}
The authors were supported by a grant from Xinjiang University under the project``Real-Variable Theory of Function Spaces and Its Applications". This work was also supported by the Natural Science Foundation of China  (No.XJEDU2020Y002 and 2022D01C407).

\end{document}